\documentclass[reqno,12pt]{amsart}

\usepackage{amssymb}
\usepackage{amssymb, latexsym}
\usepackage{hyperref}
\usepackage{amsmath}
\usepackage{amscd}
\usepackage{enumerate}
\usepackage{amsfonts}
\usepackage{graphicx}
\usepackage[all]{xypic}
\usepackage{mathrsfs}
\usepackage{tikz}

\usepackage{setspace}

\allowdisplaybreaks

\newcommand{\C}{{\mathbb{C}}}
\newcommand{\R}{{\mathbb{R}}}

\theoremstyle{plain}
\newtheorem{theorem}{Theorem}
\newtheorem{proposition}[theorem]{Proposition}
\newtheorem{lemma}[theorem]{Lemma}
\newtheorem{corollary}[theorem]{Corollary}

\theoremstyle{definition}

\newtheorem{remark}[theorem]{Remark}

\numberwithin{equation}{section}
\numberwithin{theorem}{section}

\numberwithin{equation}{section}

\newcommand{\KKK}{\mathcal{K}}

\renewcommand{\th}{{\tilde{H}_c}}

\renewcommand{\R}{\mathbb{R}}
\renewcommand{\i}{i}
\renewcommand{\o}[1]{\text{o}\left({#1}\right)}
\newcommand{\abs}[1]{\left|{#1}\right|}
\newcommand{\sts}[1]{\left({#1}\right)}
\newcommand{\ltl}[1]{~\left\{{#1}\right\}~}
\newcommand{\dual}[2]{\left\langle{#1}~,~{#2}\right\rangle}

\newcommand{\normth}[1]{\|#1\|_{\tilde{H}_c}}
\newcommand{\normx}[1]{\|#1\|_{X_c}}

\newcommand{\normdh}[1]{\|#1\|_{\dot{H}^1}}
\newcommand{\normf}[1]{\|#1\|_{L^4}}
\newcommand{\normsix}[1]{\|#1\|_{L^6}}
\newcommand{\normto}[1]{\|#1\|_{L^2}}

\makeatletter

\newcommand{\Rmnum}[1]{\expandafter\@slowromancap\romannumeral #1@}
\makeatother

\begin{document}
  \onehalfspacing

\title[DNLS]
{Solitary waves for nonlinear Schr\"{o}dinger equation with
derivative}

\author[]{Changxing Miao}
\address{\hskip-1.15em Changxing Miao:
\hfill\newline Institute of Applied Physics and Computational
Mathematics, \hfill\newline P. O. Box 8009,\ Beijing,\ China,\
100088.}  \email{miao\_changxing@iapcm.ac.cn}

\author[]{Xingdong Tang}
\address{\hskip-1.15em Xingdong Tang \hfill\newline Institute of
Applied Physics and Computational Mathematics, \hfill\newline P. O.
Box 8009,\ Beijing,\ China,\ 100088.}
\email{xdtang202@163.com}

\author[]{Guixiang Xu}
\address{\hskip-1.15em Guixiang Xu \hfill\newline Institute of
Applied Physics and Computational Mathematics, \hfill\newline P. O.
Box 8009,\ Beijing,\ China,\ 100088. }
\email{xu\_guixiang@iapcm.ac.cn}

\subjclass[2010]{35L70, 35Q55}

\keywords{Derivative Schr\"{o}dinger equation; Global
well-posedness; Invariant set; Solitary waves; Structure analysis; Variational method.}

\begin{abstract}In this paper, we characterize a family of solitary waves for  NLS with derivative (DNLS) by the structue analysis and the variational argument. Since (DNLS) doesn't enjoy the Galilean invariance any more, the structure analysis here is closely related with the nontrivial momentum and shows the equivalence of nontrivial solutions between the quasilinear and the semilinear equations. Firstly, for the subcritical parameters $4\omega>c^2$ and the critical parameters $4\omega=c^2, c>0$,  we show the existence and uniqueness of the solitary waves for (DNLS), up to the phase rotation and spatial translation symmetries. Secondly, for the critical parameters $4\omega=c^2, c\leq 0$ and the supercritical parameters $4\omega<c^2$, there is no nontrivial solitary wave for (DNLS). At last, we make use of the invariant sets, which is related to the variational characterization of the solitary wave, to obtain the global existence of solution for (DNLS) with initial data in the invariant set $\mathcal{K}^+_{\omega,c}\subseteq H^1(\R)$, with $4\omega=c^2, c>0$ or $4\omega>c^2$.

On one hand, different with the scattering result for the $L^2$-critical NLS in \cite{Dod:NLS_sct}, the scattering result of (DNLS) doesn't hold for initial data in $\mathcal{K}^+_{\omega,c}$ because of the existence of infinity many small solitary/traveling waves in $\mathcal{K}^+_{\omega,c},$ with $4\omega=c^2, c>0$ or $4\omega>c^2$. On the other hand, our global result improves the global result in \cite{Wu-DNLS, Wu-DNLS2} (see Corollary \ref{cor:gwp}).
\end{abstract}

\maketitle


%
%
%
%
\section{Introduction}
In this paper, we consider the solitary waves of nonlinear
Schr\"{o}dinger equation with derivative

\begin{equation} \label{DNLS}
\left\lbrace \aligned
    &
    \i\partial_t u + \partial^2_x u  + \frac{1}{2}\i \abs{u}^2\partial_x u - \frac{1}{2}\i u^2\partial_x\overline{u}+\frac{3}{16}\abs{u}^4 u  =0,\; t\in \R\\
    &
    u\sts{0,x}=u_0\sts{x}\in H^1(\R),
\endaligned
\right.
\end{equation}
the equation \eqref{DNLS} appears in plasma physics \cite{MOMT-PHY, M-PHY,
SuSu-book}, and has many equivalent forms. For example, it is equivalent to the following equation
\begin{equation} \label{DNLS-a}
\left\lbrace \aligned
    &
    \i\partial_t v + \partial^2_x v + \i\partial_x\sts{\abs{v}^2 v}=0,\; t\in \R \\
    &
    v\sts{0,x}=v_0\sts{x}\in H^1(\R)
\endaligned
\right.
\end{equation}
by the following gauge
transformation
\begin{align*}v(t,x)\mapsto u(t,x)=G_{3/4}(v)(t,x) := e^{\i \frac34
\int^x_{-\infty}|v(t,\eta)|^2\;d\eta}v(t,x). \end{align*}

The equation \eqref{DNLS} is $L^2$-critical derivative NLS since the scaling transformation
\begin{align*}u(t,x)\mapsto u_{\lambda}(t,x)=\lambda^{1/2}u(\lambda^2t, \lambda x)
\end{align*}
leaves both \eqref{DNLS} and the mass invariant. The mass, momentum
and energy of the solution for \eqref{DNLS} are defined as following
\begin{align}
M(u)(t)= & \frac12 \int |u(t,x)|^2 \;dx, \label{def:mass}\\
P(u)(t)=& -\frac12 \Im \int\bar{u}\,\partial_x u + \frac18 \int
|u(t,x)|^4\; dx, \label{def:mom}\\
E(u)(t)=& \frac12 \int |\partial_x u(t,x)|^2 \;dx -\frac{1}{32} \int
|u(t,x)|^6\; dx. \label{def:eng}
\end{align}
They are conserved under the flow \eqref{DNLS} by the local
well-posedness theory in $H^1$ according to the phase rotation,
spatial translation and time translation invariances. Since \eqref{DNLS} or \eqref{DNLS-a} doesn't enjoy the
Galilean and pseudo-conformal invariance any more, there is no explicit blowup solution for \eqref{DNLS} and the momentum is not trivial in dealing with the solitary/traveling waves of \eqref{DNLS} any more.

Local well-posedness thery for \eqref{DNLS} in the energy space was
worked out by N. Hayashi and T. Ozawa \cite{HaOz-94, Oz-96}. They
combined the fixed point argument with $L^4_IW^1_{\infty}(\R)$
estimate to construct the local-in-time solution with arbitrary data in the
energy space. For other results, we can refer to \cite{Ha-93,
HaOz-92}. Since \eqref{DNLS} is $\dot H^1$-subcritical case, the
maximal lifespan interval of the energy solution  only depends  on the $H^1$ norm of
initial data.
\begin{theorem}\label{Thm:LWP}\cite{HaOz-94, Oz-96} For any $u_0 \in H^1(\R)$ and $t_0 \in \R$, there
exists a unique maximal-lifespan solution $u:I\times \R \rightarrow
\C$ to \eqref{DNLS} with $u(t_0)=u_0$, the map $u_0\rightarrow u$ is continuous from $H^1(\R)$ to $C(I, H^1(\R))\cap L^4_{loc}(I; W^{1,\infty}(\R))$. Moreover the solution has the
following properties:
\begin{enumerate}
\item $I$ is an open neighborhood of $t_0$.

\item The mass, momentum and energy are conserved, that is, for all
$t\in I$,
\begin{align*}
M(u)(t)=M(u)(t_0), \;\; P(u)(t)=P(u)(t_0),\;\; E(u)(t)=E(u)(t_0).
\end{align*}

\item If\; $\sup(I)<+\infty$, or $( \inf(I)>-\infty)$ , then
\begin{align*}
\lim_{t\rightarrow\sup(I)}\big\|\partial_x
u(t)\big\|_{L^2}=+\infty,\;\;
\left(\lim_{t\rightarrow\inf(I)}\big\|\partial_x
u(t)\big\|_{L^2}=+\infty, respectively.\right)
\end{align*}

\item If $\big\| u(0)\big\|_{H^1}$ is sufficiently small,
then $u$ is a global solution.
\end{enumerate}
\end{theorem}

The sharp local well-posedness result in $H^s, s\geq 1/2$ is due to
H. Takaoka \cite{Ta-99} by Bourgain's Fourier restriction method.
The sharpness is shown in \cite{Ta-01} in the sense that nonlinear
evolution $u(0)\mapsto u(t)$ fails to be $C^3$ or even uniformly
$C^0$ in this topology, even when $t$ is arbitrarily close to zero
and $H^s$ norm of the data is small(see also Biagioni-Linares
\cite{BiLi-01-Illposed-DNLS-BO}).

In \cite{Oz-96}, the global well-posedness is obtained for \eqref{DNLS}
in energy space under the smallness condition
\begin{align}\label{Cond:smalldata}
\|u_0\|_{L^2} < \sqrt{2 \pi},
\end{align}
the argument is based on the sharp Gagliardo-Nerenberg inequality
and the energy method (conservation of mass and energy). This is
improved by H. Takaoka \cite{Ta-01}, who proved global
well-posedness in $H^s$ for $s>32/33$ under the condition
\eqref{Cond:smalldata}. His argument is based on Bourgain's
restriction method, which separated the evolution of low frequencies
and of high frequencies of initial data and notices that nonlinear
evolution has $H^1$ regularity effect even for rough solution $u\in
H^s$. In \cite{CKSTT-01, CKSTT-02}, I-team used the ''I-method'' to
show global well-posedness in $H^s, s>1/2$ under
\eqref{Cond:smalldata}, I-team defined $Iu$ as a modified $H^s$
norm, whose energy is nearly conserved in time by capturing
nonlinear cancellation in frequency space under the flow
\eqref{DNLS}. Later, Miao, Wu and Xu \cite{MiaoWX-2011} showed the
sharp global well-posedness in $H^{1/2}$ under
\eqref{Cond:smalldata} by using I-method and the refined resonant
decomposition technique.

In this paper, we consider the existence of the solitary/traveling waves
in the energy space for \eqref{DNLS} and its role in the long time analysis
of solution to \eqref{DNLS}. It is known in \cite{SaarloosH:PhyD}
that \eqref{DNLS} has a two-parameter family of solitary/traveling
waves solutions of the form:
\begin{align}\label{introduce:tw}u_{\omega,c}(t,x)= &  e^{i\omega t } \varphi_{\omega,c}(x-ct) \\
:=& e^{i\omega t + \i \frac{c}{2}(x-ct)}\phi_{\omega,c}(x-ct)
\label{introduce:structure}
\end{align}where $(\omega,c)\in \R^2, \; 4\omega>c^2$ and

\begin{align}\label{solt:sub}\phi_{\omega,c}(x)=\left[\frac{\sqrt{\omega}}{4\omega-c^2}
\left\{\cosh\big(\sqrt{4\omega-c^2}x\big)-\frac{c}{2\sqrt{\omega}}\right\}\right]^{-1/2},
\end{align}
which is a positive solution of
\begin{align}\label{V:inner para}\left(\omega-\frac{c^2}{4}\right)\phi - \partial^2_x \phi
-\frac{3}{16}\left|\phi\right|^4\phi = -\frac{c}{2}
\left|\phi\right|^2\phi.
\end{align}

Note that the solitary/traveling waves have the following mass
\begin{align}\left\| e^{i\omega t }
\varphi_{\omega,c}(x-ct)\right\|^2_{L^2}= & \left\|e^{i\omega t + \i
\frac{c}{2}(x-ct)} \phi_{\omega,c}(x-ct) \right\|^2_{L^2} \nonumber\\
= & 8 \tan^{-1}\sqrt{ \frac{2 \sqrt{\omega} +c }{2 \sqrt{\omega} -c
}}.\label{mass:tw}
\end{align}
This implies
\begin{align*}
\lim_{4\omega >c^2, (\omega,c)\rightarrow(1,2)} \left\| e^{i\omega t
} \varphi_{\omega,c}(x-ct)\right\|_{L^2} = \sqrt{4 \pi}.
\end{align*}

As for $(\omega,c)=(1,0)$, the role of the momentum in
\eqref{V:inner para} disappears. In addition, we have $E\left( e^{it
} \varphi_{1,0}\right)=0$ and $\left\| e^{it }
\varphi_{1,0}(x)\right\|_{L^2}=\sqrt{2\pi}$, which corresponds to
the condition \eqref{DNLS} and sharp Gagliardo-Nirenberg inequality
in \cite{Wein}. As for $  4\omega > c^2 $, we have
$E\left(e^{i\omega t } \varphi_{\omega,c}(x-ct)\right)<0$ for $c>0$,
and $E\left(e^{i\omega t } \varphi_{\omega,c}(x-ct)\right)>0$ for
$c<0$, Colin and Ohta prove its stability by the variational
method (the concentration compactness argument) in
\cite{ColinOhta-DNLS}. For the special case $4\omega > c^2 $ with
$c<0$, we can refer to \cite{GuoWu:orbitalStab}.

As shown above,  $e^{it } \varphi_{1,0}$, which corresponds to
$(\omega,c)=(1,0)$, is not the unique solitary wave of
\eqref{DNLS}, up to the phase rotation and spatial translation
symmetries. In \cite{Wu-DNLS}, Wu showed that there exists a small
$\epsilon_*>0$, such that the solution $u$ of \eqref{DNLS} globally
exists under the condition \[ \|u_0\|_{L^2} <
\big\|\varphi_{1,0}\big\|_{L^2} + \epsilon_* = \sqrt{2 \pi} +
\epsilon_*.\]

It is the aim to characterize the solitary waves and show its role in the long time
analysis of solution for \eqref{DNLS} from the point of view in \cite{PaySatt:Instab}. In order to do so, we firstly
give the variational characterization of solitary waves.  Now, we
consider the solitary solutions for \eqref{DNLS} with the following
form:
   \begin{equation*}
     u\sts{t,x}=e^{\i \omega t}\varphi_{\omega,c}\sts{x-c t}.
   \end{equation*}
    It is easy to verify that $\varphi_{\omega,c}$ satisfies
    \begin{align}\label{V:EI:bad}
       \omega\varphi -\partial^2_x \varphi
        -\frac{3}{16}\abs{\varphi}^4\varphi=-\i c\partial_x \varphi + \frac{1}2 \i  |\varphi|^2\partial_x\varphi -\frac{1}2 \i \varphi^2\partial_x
        \bar{\varphi}.
    \end{align}

Note that the term $\displaystyle -c\i \partial_x \varphi +
\frac{i}2 |\varphi|^2\partial_x\varphi -\frac{i}2
\varphi^2\partial_x
        \bar{\varphi}$
is not compatible with the momentum.  While, after the key structure analysis of solution in Section
\ref{sect:structure:compactness}, we find that
        \eqref{V:EI:bad} is equivalent to the following
\begin{equation}\label{V:EI-b}
 \omega\varphi -\partial^2_x\varphi  -\frac{3}{16}\abs{\varphi}^4\varphi=-\i c\partial_x \varphi
 -\frac{c}{2}
 \abs{\varphi}^2\varphi,
\end{equation}
which is compatible with the mass, momentum and energy, and the
solution of \eqref{V:EI-b} is the critical point of
\begin{align}\label{Funct:J}J_{\omega,c}(\varphi):=E(\varphi)+\omega M(\varphi)+c P(\varphi)
\end{align} in $H^1(\R)$. More precisely, we have

\begin{theorem}\label{Thm:Existence of TW}Let\footnote{For the subcritical parameters $4\omega>c^2$, $\phi_{\omega,c}(x)$ decays exponentially, while for the critical
parameters $4\omega=c^2, c>0$, $\phi_{\omega,c}(x)$ decays
polynomially.}
\begin{equation*}
\phi_{\omega,c}(x)= \left\lbrace \aligned
\left[\frac{\sqrt{\omega}}{4\omega-c^2}
\left\{\cosh\big(\sqrt{4\omega-c^2}x\big)-\frac{c}{2\sqrt{\omega}}\right\}\right]^{-1/2},
& 4\omega > c^2,
\\
 2\sqrt{c}\cdot \left(c^2x^2+1\right)^{-1/2}, & 4\omega=c^2, c>0.
\endaligned
\right.
\end{equation*}
Then the following results hold
\begin{enumerate}
\item For the subcritical case $4\omega>c^2$. $\varphi_{\omega,c}(x)=e^{i\frac{c}{2}x} \phi_{\omega,c}(x) $ is a unique
solution of \eqref{V:EI:bad} in $H^1(\R, \C)$, up to the phase
rotation and spatial translation symmetries of \eqref{V:EI:bad}.

\item For the critical case $4\omega=c^2, c>0$. $\varphi_{\omega,c}(x)=e^{i\frac{c}{2}x} \phi_{\omega,c}(x)$ is a unique
solution of \eqref{V:EI:bad} in $H^1(\R, \C)$, up to the phase
rotation and spatial translation symmetries of \eqref{V:EI:bad}.

\item  For the critical case $4\omega=c^2, c\leq0$ and the supercritical case $4\omega < c^2$. \eqref{V:EI:bad}
has no nontrivial solution in $H^1(\R, \C)$.
\end{enumerate}
\end{theorem}

\begin{remark}
We make some remarks on the above result.
\begin{enumerate}
\item We have the following pointwise convergence.
\begin{align*}
\varphi_{\omega,c}(x) \rightarrow &\; 0\;\; \text{as}\;\; 4\omega
> c^2, (\omega,c) \rightarrow (1,-2),\\
\varphi_{\omega,c}(x) \rightarrow &\; 2\sqrt{2}e^{ix}/
\left(4x^2+1\right)^{1/2}\;\; \text{as}\;\; 4\omega
> c^2, (\omega,c) \rightarrow (1,2).
\end{align*}


\item After the structure analysis (Lemma \ref{lem:structure}), we can also obtain the existence of the solution for \eqref{V:inner para} by ODE argument, (See [Theorem 5, \cite{BerestLions:NLS:GS}]). Here we use the variational argument to show the existence of solution for \eqref{V:EI:bad} (which is equivalent to the existence of solution for \eqref{V:inner para} by the structure analysis), its advantage is that we can show the global existence of the energy solution for \eqref{DNLS} in some invariant set $\mathcal{K}^{+}_{\omega,c}$ in Theorem \ref{Thm:GWP} by the local wellposedness result and the variational argument.

\item The variational characterization of the solitary waves with $4\omega > c^2$ in
\cite{ColinOhta-DNLS} doesn't work for the critical case $4\omega =
c^2, c>0$ as well, we can refer to Lemma 7 in \cite{ColinOhta-DNLS}.
Here we use the structure analysis of solution to show the
equivalence of nontrivial solution between \eqref{V:EI:bad} and
\eqref{V:EI-b}. After showing this property, it is easy to use the
variational method \cite{AmbMal, Willem:MinMax} to show the
existence of $\varphi_{\omega, c}$ to \eqref{V:EI-b} in $X_c$ space
with structure\footnote{See definition in \eqref{Def:space_crit}}. The uniqueness (up to the phase rotation and spatial
translation symmetries)\footnote{The uniqueness is obtained by the
standard ODE argument due to one dimensional spatial variable.} and
$H^1$ regularity of the solitary wave imply the existence and
uniqueness of the minimizer of the variational problem in the energy
space.

\item About the stability of the sum of two solitary waves of \eqref{DNLS} with subcritical parameters in the energy space, we can refer to \cite{MTXu-DNLS-TWSta}, which is obtained by the linearized argument, modulational stability analysis and the energy method. Recently, we have learned that the stability of the sum of k solitary waves of (DNLS) has been obtained independently by Le Coz and Wu \cite{LeW}.
\end{enumerate}
\end{remark}

Secondly, we can consider the role of the solitary waves $e^{\i \omega
t}\varphi_{\omega,c}\sts{x-ct}$ in the long time analysis of solution to
\eqref{DNLS}. We can refer to \cite{NakSchlag:book, PaySatt:Instab}. For the subcritical case $4\omega > c^2$ or the
critical case $4\omega = c^2$ with $c>0$,  we let
$J^0_{\omega,c}=J_{\omega,c}\left(\varphi_{\omega,c}\right)$, and
introduce the functional $K_{\omega,c}\sts{\varphi}$, which is the
invariant quantity of solutions to \eqref{V:EI-b}
\begin{align*}
K_{\omega,c}\sts{\varphi}:=
   \int\sts{\abs{\varphi_x}^2-\frac{3}{16}\abs{\varphi}^6+\omega\abs{\varphi}^2  - c\Im\sts{\overline{\varphi}\varphi_x}+\frac{c}{2}\abs{\varphi}^4
  }dx,
\end{align*} and two subsets
in the energy space $H^1$
 \begin{align*}
\KKK^+_{\omega,c} : = & \left\{\varphi\in H^1:
J_{\omega,c}\big(\varphi\big) < J^0_{\omega,c} , \;\;
K_{\omega,c}\big(
\varphi \big) \geq 0\right\},   \\
\KKK^-_{\omega,c} : = & \left\{\varphi\in H^1:
J_{\omega,c}\big(\varphi\big) < J^0_{\omega,c}, \;\;
K_{\omega,c}\big(\varphi \big) < 0\right\}.
 \end{align*}

As a consequence of the variational characterization of the solitary
waves and the local well-posedness theory to \eqref{DNLS}, we have

\begin{theorem}\label{Thm:GWP}The following results hold
\begin{enumerate}
\item For $4\omega>c^2$ or $4\omega=c^2$ with $c>0$,  we have $\KKK^{\pm}_{\omega,c} \neq \emptyset$,
and they are invariant sets under the flow of \eqref{DNLS} in
$H^1(\R, \C)$.
\item Let $u(0)\in H^1 $, and $u$ be the solution of \eqref{DNLS}
with initial data $u(0)$ and $I$ be its maximal interval of
existence. Then if $u(0) \in \KKK^{+}_{\omega,c}$ for some $(\omega,
c)$ with $4\omega>c^2$ or $4\omega=c^2, c>0$, then $I=\R$.
\end{enumerate}
\end{theorem}

\begin{remark}
\begin{enumerate}
\item For sufficiently small $\epsilon>0$, we have $B_{\epsilon}(0)\subseteq
\KKK^{+}_{\omega,c}$ for $4\omega>c^2$ or $4\omega=c^2$ with $c>0$.

\item
What will happen for the solution of \eqref{DNLS} with initial data in $\KKK^+_{\omega, c}$? In fact,
for $4\omega_1 > c_1^2$ with $| (\omega_1,c_1) - (1,-2)|\ll 1 $ and
$4\omega_2  \geq c_2^2 $ with $| (\omega_2,c_2) - (1,2)|\ll 1 $, we
have $M(\varphi_{\omega_1, c_1})+
\frac{c_1}{2}P(\varphi_{\omega_1,c_1})\ll 1$ and
\[\varphi_{\omega_1,c_1} \in \KKK^+_{\omega_2, c_2}, \]
which means that there are infinity many small solitary/traveling waves in
$\KKK^{+}_{\omega_2, c_2}$, therefore the scattering result for
\eqref{DNLS} with initial data in $\KKK^{+}_{\omega_2, c_2}$ doesn't hold any more (See Figure 1). This is a significant difference with the $L^2$-critical NLS in \cite{Dod:NLS_sct}.

\item For $4\omega>c^2$ or $4\omega=c^2$ with $c>0$,
 we have no long time analysis for solutions with initial data in
$\KKK^-_{\omega,c}$  since there is no effective Virial identity for the energy solution of \eqref{DNLS}.

\item In \cite{FukHI:gDNLS}, Fukaya, Hayashi and Inui  obtained the
analogous global result for the generalized derivative nonlinear Schr\"{o}dinger equation a few months after we submitted our paper.
\end{enumerate}
\end{remark}

\begin{center}
 \begin{tikzpicture}[scale=0.5]
    \begin{scope}[xshift=-4cm,yshift=-4cm]
        \begin{scope}[nearly opaque,fill=green!0!white,line width=1pt]
            \filldraw
             (-5,0) .. controls(-3,4) and (3,4) .. (5,0);
            \filldraw
             (-5,0) .. controls(-3,-4) and (3,-4) .. (5,0);
            \draw[loosely dashed](-9,3) to [out=-15,in= 115] (-5,0) to [out=-115,in=0] (-9,-3);
            \draw[loosely dashed](9,3) node[right] {$\scriptstyle J_{1,2}=J_{1,2}^0$} to [out=-165,in=65](5,0) to
                [out=-65,in=165] (9,-3);
            \draw[densely dashed] (-4,-5)
                to [out=125,in=-95] (-5,0) to [out=85,in=-125](-4,5);
            \draw[densely dashed] (4,-5)node[below] {$\scriptstyle K_{1,2}=0$}
                to [out=55,in=-95] (5,0) to [out=95,in=-55](4,5);
            \draw[opaque] (8,-0.5) node {$\scriptstyle J_{1,2}<J_{1,2}^0$};
            \draw[opaque] (7.8,-1.5) node {$\scriptstyle K_{1,2}<0$};
            \draw[opaque] (3.5,-0.5) node {$\scriptstyle J_{1,2}<J_{1,2}^0$};
            \draw[opaque] (3.2,-1.5) node {$\scriptstyle K_{1,2}>0$};
            \fill[opaque,black] (5,0) circle (.07cm)node[right] {$\varphi_{1,2}$};
        \end{scope}
        \begin{scope}[semitransparent,blue!60!white,rotate=60,scale=0.8,line width=1pt]
            \filldraw
             (-5,0) .. controls(-3,4) and (3,4) .. (5,0);
            \filldraw
             (-5,0) .. controls(-3,-4) and (3,-4) .. (5,0);
            \draw[loosely dashed]
                (-9,3) to [out=-15,in= 115] (-5,0) to [out=-115,in=0] (-9,-3);
            \draw[loosely dashed] (9,3) to [out=-165,in=65](5,0) to [out=-65,in=165] (9,-3);
            \fill[opaque,color=blue] (5,0) circle (.1cm)node[above] {$Q$};
        \end{scope}
        \begin{scope}[semitransparent,red!80!white,line width=1pt,rotate=120,scale=0.6425]
            \filldraw
             (-5,0) .. controls(-3,4) and (3,4) .. (5,0);
            \filldraw
             (-5,0) .. controls(-3,-4) and (3,-4) .. (5,0);
            \draw[loosely dashed]
                (-9,3) to [out=-15,in= 115] (-5,0) to [out=-115,in=0] (-9,-3);
            \draw[loosely dashed] (9,3) to [out=-165,in=65](5,0) to [out=-65,in=165] (9,-3);
            \fill[opaque,red] (5,0) circle (0.12cm)node[above left] {$A$};
        \end{scope}
        \begin{scope}[nearly transparent,black,line width=1pt,rotate=200,scale=0.5]
            \filldraw
             (-5,0) .. controls(-3,4) and (3,4) .. (5,0);
            \filldraw
             (-5,0) .. controls(-3,-4) and (3,-4) .. (5,0);
            \draw[loosely dashed]
                (-9,3) to [out=-15,in= 115] (-5,0) to [out=-115,in=0] (-9,-3);
            \draw[loosely dashed]
                (9,3) to [out=-165,in=65](5,0) to [out=-65,in=165] (9,-3);
            \fill[opaque,color=black] (5,0) circle (.12cm);
        \end{scope}
        \filldraw[fill=white] (0,0) circle(0.06cm) node[right]{$0$};
        \node [below=4cm, align=flush center,text width=8cm,black] at (0,0)
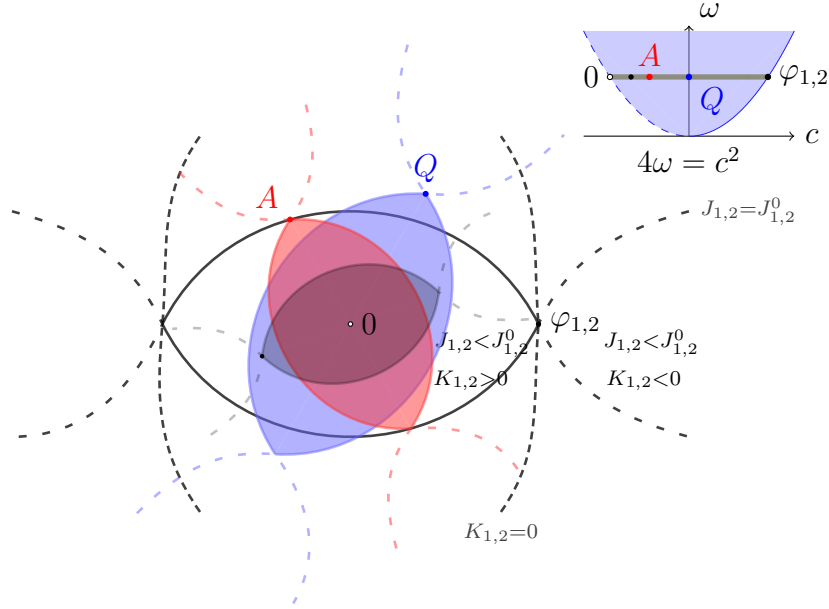
        {Figure 1: $\varphi_{\omega,c}$ where $Q=\varphi_{1,0}, A\approx \varphi_{1,-1.1}$
        };
    \end{scope}
        \begin{scope}[xshift=5cm, yshift = 1cm, scale=0.7]
        \draw[->] (-4,0) -- (4,0) node[right] {$c$};
        \draw[->] (0,0) -- (0,4.2);
        \draw (0,4.1) node[above right] {$\omega$};
        \draw[domain=0:4,color=blue] plot (\x,0.25*\x*\x) ;
        \draw[domain=-4:0,color=blue,dashed] plot (\x,0.25*\x*\x);
        \draw[color=green!50!red,line width=2pt] (-3,0.25*3*3)--(3,0.25*3*3);

        \fill[semitransparent, domain=0:4,color=blue!40!white] plot (\x,0.25*\x*\x) ;
        \fill[semitransparent, color=blue!40!white] (0,0)--(0,4)--(4,4);
         \fill[semitransparent, domain=-4:0,color=blue!40!white] plot (\x,0.25*\x*\x) ;
        \fill[semitransparent, color=blue!40!white] (0,0)--(0,4)--(-4,4);

        \filldraw (3,0.25*3*3) circle (.1cm) node [right] {$\varphi_{1,2}$};
        \filldraw[fill=white] (-3,0.25*3*3) circle (.1cm) node [left] {$0$};
        \filldraw[blue] (0,0.25*3*3) circle (0.1cm) node [below right] {$Q$};
        \filldraw[red] (-1.5,0.25*3*3) circle (0.1cm) node [above] {$A$};
        \filldraw (-2.2,0.25*3*3) circle (0.08cm) ;

        \draw (0,0) node[below] {$4\omega=c^{2}$};
    \end{scope}
    \end{tikzpicture}
    \end{center}

As a corollary, we can obtain the following result (See also Corollary 1.4 in \cite{FukHI:gDNLS}).
\begin{corollary}\label{cor:gwp} Let $u(0,\cdot)=u_0(\cdot)\in H^1(\R) $, and satisfy one of the following conditions
 \begin{enumerate}
 \item $M(u_0)<2\pi$,

 \item $M(u_0)=2\pi$ and $P(u_0)<0$,

 \item $M(u_0)=2\pi$ and $P(u_0)=0$ and $E(u_0)<0$.
 \end{enumerate}
Then the solution to \eqref{DNLS} exists globally  in $H^1(\R)$.
\end{corollary}
\begin{remark}
This result can improve the global result in \cite{Wu-DNLS, Wu-DNLS2}. In fact, we can show that the subset of $H^1(\R)$ with the property (3) is empty by the variational characterization of the solitary wave, this phenomena is similar as that for the $L^2$-critical NLS in \cite{Wein:singul}.
\end{remark}

Throughout this paper, we will use the following notations. The
tempered distribution is denoted by $\mathcal{S}'(\R^n)$. We use $A\lesssim B$
to denote an estimate of the form $A\leq CB$ for some constant $C$.
If $A\lesssim B$ and $B\lesssim A$, we say that $A\approx B$.

At last, this paper is organized as follows. In Section
\ref{sect:structure:compactness}, we first give the structure
analysis of solution to \eqref{V:EI:bad}, then show the variational
characterization of the solitary waves in space  $\th$ for the
subcritical parameters and $X_c$ for the critical parameters with
structure, and obtain the threshold $J^0_{\omega,c}$ in terms of
the solitary waves in Section \ref{sect:vc:subcritical} and Section
\ref{sect:vc:critical}, respectively; In Section \ref{sect:GWP}, we
make use of the variational characterization of the solitary waves and
the local wellposedness of \eqref{DNLS} to prove Theorem \ref{Thm:GWP} and Corollary \ref{cor:gwp}.

%
%
%
%

\section{Existence and nonexistence of traveling
waves}\label{sect:exit}
In this section, we firstly consider the existence of the solitary/traveling
waves for \eqref{DNLS} with the following form:
   \begin{equation*}
     u\sts{t,x}=e^{\i \omega t}\varphi_{\omega,c}\sts{x-c t}.
   \end{equation*}
    It is easy to check that $\varphi_{\omega,c}$ satisfies
    \begin{align}\label{Eq:varphi:I}
       \omega\varphi -\partial^2_x \varphi
        -\frac{3}{16}\abs{\varphi}^4\varphi=-c\i \partial_x \varphi + \frac{1}2 \i  |\varphi|^2\partial_x\varphi -\frac{1}2 \i \varphi^2\partial_x
        \bar{\varphi}.
    \end{align}

\subsection{Structure analysis, nonexistence and compactness result}
\label{sect:structure:compactness}Although the left hand side in \eqref{Eq:varphi:I} consists with the definitions of the mass and energy in \eqref{def:mass} and \eqref{def:eng}, while the right hand side
is not compatible with the definitions of the momentum in \eqref{def:mom}. This motivates us
to explore more properties about the solitary waves. Here we make use of the
structure of the solitary waves. Note that $\varphi\in
H^1(\R, \C)\setminus\ltl{0}$ is a nontrivial solution to
\eqref{Eq:varphi:I} with the structure
$\varphi\sts{x}:=e^{\i\frac{c}{2}x}\phi\sts{x}$, if and only if
$\phi\in H^1(\R, \C)\setminus\ltl{0}$ satisfies
\begin{equation}\label{Eq:phi:I}
  \sts{\omega-\frac{c^2}{4}}\phi-\partial^2_x \phi
    -\frac{3}{16}\abs{\phi}^4\phi
  = -\frac{c}{2}\abs{\phi}^2\phi
    + \frac{1}2 \i  |\phi|^2\partial_x\phi -\frac{1}2 \i
    \phi^2\partial_x  \bar{\phi}
\end{equation}
For this equation, we have

\begin{lemma}\label{lem:structure}$\phi\in H^1(\R, \C)$ is a nontrival solution to
\eqref{Eq:phi:I} if and only if $\phi\in H^1(\R, \C)$ satisfies
\begin{equation}\label{Eq:phi:II}
  \sts{\omega-\frac{c^2}{4}}\phi-\partial^2_x \phi
    -\frac{3}{16}\abs{\phi}^4\phi
  = -\frac{c}{2}\abs{\phi}^2\phi.
\end{equation}
\end{lemma}
\begin{proof}
See Lemma 2 in \cite{ColinOhta-DNLS}.
\end{proof}

\begin{remark}\label{rem:rvs}By the proof of Theorem 8.1.6 in \cite{Caz:NLS:book}, we know that the solution of \eqref{Eq:phi:II} can be taken the positive, even and real valued function up to a fixed phase rotation and spatial translation, from which we can take the solution $\phi$ of \eqref{Eq:phi:II} to be a real function, that is $\phi\in H^1(\R, \R)$.
\end{remark}

Now we divide $(\omega,c)\in \R^2$ into several regions.
\begin{enumerate}
\item the supercritical case: $4\omega<c^2$;
\item the critical case: $4\omega=c^2$;
\item the subcritical case: $4\omega>c^2$.
\end{enumerate}

\begin{proposition}\label{nonexist:crit case} For the supercritical case $4\omega < c^2$ and the critical case $4\omega = c^2, c\leq 0$, \eqref{Eq:varphi:I} has no nontrivial solution in $H^1(\R, \C)$.
\end{proposition}
\begin{proof}
After the structure analysis in Lemma \ref{lem:structure} and Remark \ref{rem:rvs}, we only need to show the nonexistence of the real valued nontrival solution to \eqref{Eq:phi:II}, which can be obtained by Theorem 5 in \cite{BerestLions:NLS:GS}.
\end{proof}

Now we consider the subcritical case $4\omega>c^2$ and the
critical case $4\omega=c^2$, $c>0$. The special structure for $\phi$
implies the special structure for $\varphi$ to \eqref{Eq:phi:I},
which induces that nontrivial solution $\varphi$ to \eqref{Eq:phi:I}
is just the nontrivial solution $\varphi$ to
\begin{equation}\label{Eq:varphi:II}
 \omega\varphi -\partial^2_x\varphi  -\frac{3}{16}\abs{\varphi}^4\varphi=-c\i\partial_x \varphi -\frac{c}{2}\abs{\varphi}^2\varphi,
\end{equation}
which exactly corresponds to the definitions \eqref{def:mass}-\eqref{def:eng} of the mass, the momentum and the energy. Formally, $\varphi$ is the critical points of the energy-mass
$E+\omega M$ provided that the momentum is fixed. Since the right hand side in \eqref{Eq:varphi:I} or \eqref{Eq:varphi:II} is not semilinear, but quasilinear, we need to combine the above structure analysis\footnote{From the proof, the structure analysis is the key
point and also necessary for us to show the existence of the solitary waves with the
critical parameters $4\omega=c^2, c>0$. The similar idea also appeared in showing
the existence of the solitary/traveling waves of Gross-Pitaevskii equation in
\cite{BethuelGS-GP-II, Gerard:GP, Maris:NLS}. }, the
Nehari manifold argument in \cite{AmbMal, Willem:MinMax} and the
symmetric-decreasing rearrangement in \cite{LiebL:book} to show the
existence of the solitary waves. It also helps to give the long time analysis
of solution to \eqref{DNLS} in next section. By the classical argument,
$\varphi\in H^1(\R, \C)\setminus\ltl{0}$ solves \eqref{Eq:varphi:II}
if and only if $\varphi\in H^1(\R, \C)\setminus\ltl{0}$ is a
nontrival critical point of the following functional
\begin{align}\label{Funct:J}
\nonumber
  J_{\omega,c}\sts{\varphi} :=&  E\sts{\varphi} +\omega M\sts{\varphi} + c P\sts{\varphi}\\
  = & \int \left(\frac{1}{2}\abs{\partial_x\varphi}^2  -\frac{1}{32}\abs{\varphi}^6 +
  \frac{\omega}{2}\abs{\varphi}^2 -
  \frac{c}{2}\Im\sts{\overline{\varphi}\partial_x\varphi}
  +\frac{c}{8}\abs{\varphi}^4\right)dx.
\end{align}
It is unbounded from below in $H^1(\R, \C)$. While, it is easy to
check that $J_{\omega,c}$ is (at least) a $C^2$ functional on
$H^1(\R, \C)$.
Moreover, as for \eqref{Eq:varphi:II}, we consider the following
quantity
\begin{align}\label{Funct:K}
K_{\omega,c}\sts{\varphi}:=
   \int\sts{\abs{\partial_x\varphi}^2-\frac{3}{16}\abs{\varphi}^6+\omega\abs{\varphi}^2  - c\Im\sts{\overline{\varphi}\partial_x\varphi}+\frac{c}{2}\abs{\varphi}^4
  }dx
\end{align}since any solution $\varphi\in H^1(\R, \C)\setminus\ltl{0}$ to
\eqref{Eq:varphi:II} satisfies
\begin{equation}
\label{eq:varphi:II:invariant}
  K_{\omega,c}\sts{\varphi}=0.
\end{equation}
Before dealing with \eqref{Eq:varphi:II}, we give a useful compactness
lemma.
\begin{lemma}
\label{lem:compact} Let $1<p<\infty$ and $\left\{ \phi_n \right\}$ be a
bounded sequence in $\dot{H}^1\sts{\R,\R}\cap L^p\sts{\R,\R}$ with
radially symmetric and nonincreasing. Then, there exists $\phi \in
\dot{H}^1\sts{\R,\R}\cap L^p\sts{\R,\R}$ with radially symmetric and
nonincreasing such that for $p<q<\infty,$ (up to a subsequence)
\begin{equation*}
\phi_n\to\phi \mbox{ strongly in } L^q\sts{\R,\R}.
\end{equation*}
\end{lemma}
\begin{proof}
Since $
  \phi_n\in \dot{H}^1\sts{\R,\R},
$ for all $n$, we have by Sobolev embedding,
\begin{equation*}
  \phi_n\in C^{\frac{1}{2}}\sts{\R,\R}.
\end{equation*}
Thus, we may assume that $\phi_n$ are well defined pointwise for all
$n.$

Since $\left\{ \phi_n \right\}$ is bounded in
$\dot{H}^1\sts{\R,\R}\cap L^p\sts{\R,\R},$ by Sobolev embedding, if
necessary up to a subsequence, there exists non-increasing, radial
function $\phi\in \dot{H}^1\sts{\R,\R}\cap L^p\sts{\R,\R}$ such that
\begin{align}
\nonumber & \phi_n\to \phi  \quad\mbox{ weakly in }
\dot{H}^1\sts{\R,\R}\cap L^p\sts{\R,\R}
\\
\nonumber & \phi_n\to \phi  \quad\mbox{ strongly in }
L^q_{\text{loc}}\sts{\R,\R}
\\
\label{app:eq:P3} & \phi_n\to \phi  \quad\mbox{ a.e. on } \mathbb{R}
\end{align}
Using $\phi\in \dot{H}^1\sts{\R,\R}\cap L^p\sts{\R,\R},$ for any
$\epsilon>0,$ there exists $R=R(\phi)>1$ large enough, which is independent of $n$ such that
\begin{equation}
\label{app:eq:P4} \phi\left( x \right) < \epsilon, \mbox{ for any }
\left\vert x \right\vert>R.
\end{equation}
Moreover, by Egorov's theorem, and  \eqref{app:eq:P3}, there exists
$E_{\epsilon}\subset \left( -2R,2R \right)$ such that
\begin{align*}\mathrm{mes}\left( \left( -2R,2R \right) \setminus
E_{\epsilon} \right)<\epsilon, \end{align*}  and
\begin{align}
\label{app:eq:P5} \phi_n\left( x \right) \to \phi\left( x \right)
\quad\mbox{  uniformlly on  } E_{\epsilon}.
\end{align}
Since $\left\{ \phi_n \right\}$ is positive, radially symmetric and
non-increasing, it follows from \eqref{app:eq:P4} and
\eqref{app:eq:P5} that for $n$ large enough,
\begin{equation*}
\phi_n\left( x \right)\leqslant \phi_n\left( x_0 \right)\leqslant
\phi\left( x_0 \right) + 2\epsilon, \quad \text{for all~~}
x\in\mathbb{R}\setminus \left( -2R,2R \right),
\end{equation*}
where $x_0\in E_{\epsilon}\cap\left( -2R,2R
\right)\setminus\sts{-R,R }$ and
$\phi\sts{x_0}<\epsilon.$ Hence, we obtain
\begin{align*}
 &  \int_{\mathbb{R}}\left\vert \phi_n\left( x \right) - \phi\left( x \right) \right\vert^q dx \\
  =& \int_{\left( -2R,2R \right)}\left\vert \phi_n\left( x \right) - \phi\left( x \right) \right\vert^q dx + \int_{\mathbb{R}\setminus \left( -2R,2R \right)}\left\vert \phi_n\left( x \right) - \phi\left( x \right) \right\vert^q dx
  \\
  =& \int_{E_{\epsilon}}\left\vert \phi_n\left( x \right)
    -\phi\left( x \right) \right\vert^q dx
    +\int_{\left( -2R,2R \right)\setminus E_{\epsilon}}
        \left\vert \phi_n\left( x \right)-\phi\left( x \right) \right\vert^q dx
  \\
    &+\int_{\R\setminus\left( -2R,2R \right)}
        \left\vert \phi_n\left( x \right)-\phi\left( x \right) \right\vert^q dx
  \\
  \leqslant &
    \int_{E_{\epsilon}}\left\vert \phi_n\left( x \right)
    -\phi\left( x \right) \right\vert^q dx
    +\int_{\left( -2R,2R \right)\setminus E_{\epsilon}}
        \left\vert \phi_n\left( x \right)-\phi\left( x \right) \right\vert^q dx
  \\
    &  +\sts{3\epsilon}^{q-p}\sup_{n}\left\Vert \phi_n \right\Vert_p^p.
\end{align*}
Then, for sufficiently large $n$, we have the result by
\eqref{app:eq:P5}, absolutely continuity of Lebesgue integral,
$\mathrm{mes}\left( \left( -2R,2R \right) \setminus E_{\epsilon}
\right)<\epsilon$, and the arbitrary smallness of $\epsilon$.
\end{proof}

\subsection{Variational characterization for the subcritical case $4\omega>c^2$}
\label{sect:vc:subcritical} In this subsection, we shall give the
variational characterization\footnote{In fact, Colin and
Ohta\cite{ColinOhta-DNLS} have given the corresponding variational
characterization via the concentration compactness argument for the
subcritical parameters $4\omega>c^2$, nevertheless, we will show
this again by the Nehari Manifold and the non-increasing
rearrangement technique. Different with Colin-Ohta's argument
\cite{ColinOhta-DNLS}, the argument here also works for the critical
parameters $4\omega=c^2, c>0$. It will be shown in Section
\ref{sect:vc:critical}.} of the solution to \eqref{Eq:varphi:II} in the subcritical case $4\omega>c^2$.
Let $J_{\omega,c}$ and $K_{\omega,c}$ denote by \eqref{Funct:J} and
\eqref{Funct:K} respectively. By Lemma \ref{lem:structure}, we will
consider the functional $J_{\omega,c}$ and $K_{\omega,c}$ in $\th.$
More precisely, we will consider the following Sobolev space with the
rotation structure
   \begin{equation*}
     \th:=\ltl{ \varphi\in\mathcal{S}'\sts{\R}
        ~:~
        \varphi\sts{x}=e^{i\frac{c}{2}x}\phi\sts{x}
        \,\, \text{with}\,\, \phi\in H^1\sts{\R,\C}},
   \end{equation*}
   with the norm
   \begin{equation*}
        \normth{\varphi}^2:=\normdh{\phi}^2 + \sts{\omega-\frac{c^2}{4}}\normto{\phi}^2,\quad\text{with } \phi\in H^1\sts{\R,\C}.
   \end{equation*}
   It is easy to verify that $\sts{\th,\normth{\cdot}}$ is
    a Hilbert space\footnote{The inner product in $\th$ is induced by the inner product in $H^1\sts{\R,\C}$, and it is homeomorphic to $H^1(\R, \C)$}.
On one hand, $K_{\omega,c}$ is well defined and of class $C^1$ on
$\th.$ On the other hand, if $\varphi\in \th\setminus\ltl{0}$ is the
solution of \eqref{Eq:varphi:II}, then $\varphi$ satisfies
\eqref{eq:varphi:II:invariant}, which implies that
$K_{\omega,c}(\varphi)$ is an invariant quantity of solutions to
\eqref{Eq:varphi:II}. Combining the above two facts, we consider the
following minimization problem
\begin{equation}\label{Sub:mini:eq:struct}
J^0_{\omega,c}=\inf\ltl{J_{\omega,c}\sts{\varphi}~:~K_{\omega,c}\sts{\varphi}=0,~~\varphi\in
\th\setminus\ltl{0}}.
\end{equation}

For convenience, we define:
\begin{align*}
  K^Q_{\omega,c}\sts{\varphi} &:= \int\sts{\abs{\partial_x\varphi}^2+\omega\abs{\varphi}^2
  -2c\Im\sts{\overline{\varphi}\partial_x\varphi}}dx, \\
  K^N_{\omega,c}\sts{\varphi} &:=K^Q_{\omega,c}\sts{\varphi}-K_{\omega,c}\sts{\varphi}
  =\int \sts{\frac{3}{16}\abs{\varphi}^6-\frac{c}{2}\abs{\varphi}^4}dx.
\end{align*}
By the definition,  we have for $\lambda>0$ and $\alpha \in
\left(\frac{c^2}{4\omega}, 1\right)$\footnote{For the critical case
$4\omega=c^2, c>0$, we need take $\alpha=1$. Notice that $
K^Q_{\omega,c}\sts{\varphi}$ is not coercive in $H^1(\R, \C)$ in the
critical case, it is just this difficulty which motivates us to
explore the structure of the solitary waves.}
\begin{align*}
K^Q_{\omega,c}\sts{\lambda\varphi}  =&  \lambda^2
\int\sts{\abs{\partial_x\varphi}^2+\omega\abs{\varphi}^2
  -2c\Im\sts{\overline{\varphi}\partial_x\varphi}}dx\\
  = & \lambda^2\left[(1-\alpha)\big\|\partial_x \varphi\|^2_{L^2}+
  \frac{1}{\alpha}\left\|\alpha\partial_x \varphi - \frac{c}{2} i\varphi\right\|^2_{L^2} +
  \left(\omega-\frac{c^2}{4\alpha}\right)\|\varphi\|^2_{L^2}
  \right],
    \end{align*} which implies that
\begin{lemma}\label{Lem:sub:smallscal}For any $\varphi\in \th \setminus\ltl{0}$, we
have \[ \lim_{\lambda \rightarrow 0 +} K^Q_{\omega,c}\sts{\lambda\varphi} =0.\]
\end{lemma}
The next lemma exhibits the behavior of $K_{\omega,c}$ near the origin of $\th.$
\begin{lemma}
\label{Lem:knear0}
  For any bounded sequence $\ltl{\varphi_n}\subset \th\setminus\ltl{0}$ with
  \begin{equation*}
    \lim_{n\to\infty}K^Q_{\omega,c}\sts{\varphi_n}=0.
  \end{equation*}
  We have for large $n,$
  \begin{equation*}
    K_{\omega,c}\sts{\varphi_n}>0.
  \end{equation*}
\end{lemma}
\begin{proof}
  Since $\varphi_n\in \th\setminus\ltl{0},$ there exists $\phi_n\in H^1\sts{\R,\C},$ such that
  $\varphi_n\sts{x}=e^{i\frac{c}{2}x}\phi_n\sts{x}$ with
   \begin{align*}
     K^Q_{\omega,c}\sts{\varphi_n} &= \int\sts{\abs{\partial_x\varphi_n}^2+\omega\abs{\varphi_n}^2
  -c\Im\sts{\overline{\varphi}_n\partial_x\varphi_n}}dx\\
     & =\int \left[\abs{\partial_x \sts{e^{-i\frac{c}{2}x}\varphi_n}}^2+\sts{\omega-\frac{c^2}{4}}\abs{e^{-i\frac{c}{2}x}\varphi_n}^2\right]dx
     \\
     &=  \int \left[\abs{\partial_x\phi_n}^2+\sts{\omega-\frac{c^2}{4}}\abs{\phi_n}^2 \right]dx.
   \end{align*}

   By $\lim_{n\to\infty}K^Q_{\omega,c}\sts{\varphi_n}=0,$ it follows from the Gagliardo-Nirenberg inequality that, for sufficiently large $n$
  \begin{align*}
       K^N_{\omega,c}\sts{\varphi_n}
        &=\int\sts{\frac{3}{16}\abs{\varphi_n}^6-\frac{c}{2}\abs{\varphi_n}^4}dx\\
        &=\int\sts{\frac{3}{16}\abs{\phi_n}^6-\frac{c}{2}\abs{\phi_n}^4}dx\\
        &\lesssim \normdh{\phi_n}^2\normto{\phi_n}^4 + \normdh{\phi_n}\normto{\phi_n}^3\\
        &=\o{ K^Q_{\omega,c}\sts{\varphi_n} },
    \end{align*}
where we used the fact $4\omega>c^2.$  Thus, for sufficiently large $n$, we have
    \begin{equation*}
        K_{\omega,c}\sts{\varphi_n}=K^Q_{\omega,c}\sts{\varphi_n}-K^N_{\omega,c}\sts{\varphi_n}\approx K^Q_{\omega,c}\sts{\varphi_n}>0.
    \end{equation*}
  This completes the proof.
\end{proof}
%

According to the aforementioned lemma, we now replace the functional
$J_{\omega,c}$ (unbounded from below) in \eqref{Sub:mini:eq:struct}
with a positive functional $H_{\omega,c}$, while extending the
minimizing region from the mountain ridge ``$K_{\omega,c}=0$'' to
the mountain flank ``$K_{\omega,c}\leq 0$''. Let
\begin{align}\label{V:H}
  H_{\omega,c}\sts{\varphi} := & J_{\omega,c}\sts{\varphi}-\frac{1}{4}K_{\omega,c}\sts{\varphi}
  \nonumber \\
  =& \int\sts{\frac{1}{4}\abs{\partial_x\varphi}^2 - \frac{c}{4} \Im\sts{\overline{\varphi}\partial_x\varphi}+\frac{\omega}{4}\abs{\varphi}^2+\frac{1}{64}\abs{\varphi}^6
  }dx \nonumber
  \\
  =&\frac{1}{4}\int\sts{{\abs{\partial_x \sts{e^{-i\frac{c}{2}x}\varphi}}^2
    +\sts{\omega-\frac{c^2}{4}}\abs{\varphi}^2}
    +\frac{1}{16}\abs{\varphi}^6}dx,
\end{align}
which is positive. According to this definition, for any $\varphi\in
\th\setminus\ltl{0} $ and $0<\lambda_1 < \lambda_2$, we have the
following monotonicity.
\begin{align}\label{H:prop}
  H_{\omega,c}\sts{\lambda_1\varphi}<  H_{\omega,c}\sts{\lambda_2\varphi}.
\end{align}
 In order to find the minimizers of \eqref{Sub:mini:eq:struct}, we
 turn to consider the following constrained minimization problem
\begin{equation}\label{Sub:mini:ineq:struct}
  \tilde{J}^0_{\omega,c}=\inf\ltl{H_{\omega,c}\sts{\varphi}~:~K_{\omega,c}\sts{\varphi}\leqslant 0,~~\varphi\in \th\setminus\ltl{0}~},
\end{equation}
The following lemma shows that two minimization problems
\eqref{Sub:mini:eq:struct} and \eqref{Sub:mini:ineq:struct} are equivalent.
\begin{lemma}
\label{lem:sub:equiminimization}
  Let $J^0_{\omega,c}$ and $\tilde{J}^0_{\omega,c}$ be defined by \eqref{Sub:mini:eq:struct} and \eqref{Sub:mini:ineq:struct} respectively.
  Then we have  \begin{enumerate}
\item   $J^0_{\omega,c}=\tilde{J}^0_{\omega,c}>0.$
\item  any minimizer for \eqref{Sub:mini:eq:struct}  is also a minimizer for \eqref{Sub:mini:ineq:struct}, and vice versa.
\end{enumerate}
\end{lemma}
\begin{proof}
  First, by definition, we have
  \begin{equation*}
    H_{\omega,c}\sts{\varphi}=J_{\omega,c}\sts{\varphi} \quad\text{for any } \varphi\in \th\setminus\ltl{0} \quad\text{ with } K_{\omega,c}\sts{\varphi}=0,
  \end{equation*}
  we have
  \begin{align*}
    J^0_{\omega,c} &= \inf\ltl{J_{\omega,c}\sts{\varphi}~:~K_{\omega,c}\sts{\varphi}=0,~~\varphi\in \th\setminus\ltl{0}~} \\
    &= \inf\ltl{H_{\omega,c}\sts{\varphi}~:~K_{\omega,c}\sts{\varphi}=0,~~\varphi\in \th\setminus\ltl{0}~}
    \\
    & \geqslant \inf\ltl{H_{\omega,c}\sts{\varphi}~:~K_{\omega,c}\sts{\varphi}\leqslant 0,~~\varphi\in \th\setminus\ltl{0}~}\\
    & =\tilde{J}^0_{\omega,c}.
  \end{align*}
Next, for any $\varphi\in \th\setminus\ltl{0}$ with
$K_{\omega,c}\sts{\varphi}<0$. By Lemma \ref{Lem:sub:smallscal} and
Lemma \ref{Lem:knear0}, there exists $\lambda_0\in\sts{0,1}$ such
that $K_{\omega,c}\sts{\lambda_0 \varphi}=0.$  The monotonicity
\eqref{H:prop} of the functional $H_{\omega, c}$ implies that
  \begin{equation*}
    J_{\omega,c}\sts{\lambda_0 \varphi} = H_{\omega,c}\sts{\lambda_0 \varphi}<H_{\omega,c}\sts{\varphi}.
  \end{equation*}
  Hence, we have $J^0_{\omega,c}\leqslant \tilde{J}^0_{\omega,c}$, which implies $(1)$.

  Next, we show $(2)$. On one hand, let $\varphi$ be any minimizer for $\tilde{J}^0_{\omega,c},$ i.e.
    \begin{equation*}
        \varphi\in \th\setminus\ltl{0} \text{ with } K_{\omega,c}\sts{\varphi}\leqslant 0\text{ and } H_{\omega,c}\sts{\varphi}=\tilde{J}^0_{\omega,c}.
    \end{equation*}
    In order to show that $\varphi$ is also a minimizer for $J^0_{\omega,c}$, we only need to show that
    $K_{\omega,c}\sts{\varphi}=0.$
 We argue by contradiction. Assume that $K_{\omega,c}\left( \varphi\right)<0,$
 by Lemma \ref{Lem:sub:smallscal} and Lemma \ref{Lem:knear0}, there exists $\lambda_0\in\sts{0,1}$ which is dependent on $\varphi$ such that
    \begin{equation*}
      K_{\omega,c}\sts{\lambda_0 \varphi}=0
    \end{equation*}
    and
    \begin{equation*}
      K_{\omega,c}\sts{\lambda \varphi}<0,\quad \text{for any } \lambda\in\left(\lambda_0,1\right].
    \end{equation*}
Thus  by the monotonicity \eqref{H:prop} of the functional
$H_{\omega,c}$, we obtain that
\begin{equation}
  \label{eq:sub:minimizer-value-lem}
  \tilde{J}^0_{\omega,c}=H_{\omega,c}\left( \varphi \right)>H_{\omega,c}\left( \lambda_0 \varphi \right)=J_{\omega,c}\left( \lambda_0\varphi \right)\geqslant J^0_{\omega,c}=\tilde{J}^0_{\omega,c},
\end{equation}
which is a contradiction. Hence,
\begin{math}
  K_{\omega,c}\left( \varphi \right)=0
\end{math}
and
\begin{math}
   \varphi
\end{math}
is also a minimizer for $J^0_{\omega,c}.$ On the other hand, let
\begin{math}
   \varphi
\end{math}
be any minimizer for $J^0_{\omega,c},$ i.e.
\begin{equation*}
    \varphi\in \th\setminus\ltl{0} \text{ with } K_{\omega,c}\sts{\varphi}=0, \text{ and } J_{\omega,c}\sts{\varphi}=J^0_{\omega,c}.
\end{equation*}
Then we have
\begin{equation*}
  \tilde{J}^0_{\omega,c}\leqslant H_{\omega,c}\left( \varphi \right)=J_{\omega,c}\left( \varphi \right)=J^0_{\omega,c}=\tilde{J}^0_{\omega,c}.
\end{equation*}
Hence,
\begin{math}
  \varphi
\end{math}
is also a minimizer for $\tilde{J}^0_{\omega,c}.$ This completes the proof.
\end{proof}

Now, we can use the non-increasing rearrangement technique in
\cite{LiebL:book} to show the existence of minimizer to
\eqref{Sub:mini:eq:struct}.

\begin{lemma}
\label{Lem:sub:exist} There exists at least one minimizer for the
minimization problem \eqref{Sub:mini:eq:struct}.
\end{lemma}
\begin{proof}
   Let $\ltl{\varphi_n}\subset \th\setminus\ltl{0}$ be a minimizing sequence of the constrained problem \eqref{Sub:mini:ineq:struct}, i.e.
  \begin{align*}
    K_{\omega,c}\sts{\varphi_n}\leqslant 0,\quad H_{\omega,c}\sts{\varphi_n}\geqslant J^0_{\omega,c} \text{~~and~~} \lim_{n\to\infty}H_{\omega,c}\sts{\varphi_n}= J^0_{\omega,c}.
  \end{align*}
By definition of $\th,$ there exists a sequence $\ltl{\phi_n}\in
H^1\sts{\R,\C}\setminus\ltl{0}$ such that
    \begin{equation*}
        ~~\varphi_n=e^{\i \frac{c}{2}x}\phi_n~~~~\text{~~and~~}
        \normth{\varphi_n}^2=\normdh{\phi_n}^2+\sts{\omega-\frac{c^2}{4}}\normto{\phi_n}^2
    \end{equation*}
    Without loss of generality, we may assume that $\phi_n$ are real valued, radially symmetric and non-increasing about the origin of $\R$.
    Indeed, for any $\psi\in \th\setminus\ltl{0}$ with $\psi=e^{\i \frac{c}{2} x}\phi$ and $\normth{\psi}^2=\normdh{\phi}^2+\sts{\omega-\frac{c^2}{4}}\normto{\phi}^2,$
    let $\phi^{\ast}$ be the Schwarz symmetrization of $\phi,$ and $\psi^{\ast}= e^{\i \frac{c}{2} x}\phi^{\ast},$
    by Schwarz rearrangement inequality in \cite[Section 7.17]{LiebL:book}, it is easy to check that
    \begin{align*}
        H_{\omega,c}\sts{\psi}&=\frac{1}{4}\int\sts{{\abs{\partial_x \sts{e^{-i\frac{c}{2}x}\psi}}^2
    +\sts{\omega-\frac{c^2}{4}}\abs{\psi}^2} +\frac{1}{16}\abs{\psi}^6}dx
        \\
        &=\frac{1}{4}\int\sts{\abs{\partial_x \phi}^2
    +\sts{\omega-\frac{c^2}{4}}\abs{\phi}^2 +\frac{1}{16}\abs{\phi}^6 }dx
        \\
        &\geqslant \frac{1}{4}\int\sts{\abs{\partial_x \phi^{\ast}}^2
    +\sts{\omega-\frac{c^2}{4}}\abs{\phi^{\ast}}^2 +\frac{1}{16}\abs{\phi^{\ast}}^6 }dx
    \\
    &=\frac{1}{4}\int\sts{{\abs{\partial_x \sts{e^{-i\frac{c}{2}x}\psi^{\ast}}}^2
    +\sts{\omega-\frac{c^2}{4}}\abs{\psi^{\ast}}^2} +\frac{1}{16}\abs{\psi^{\ast}}^6}dx
    \\
             &=  H_{\omega,c}\sts{\psi^{\ast}},
    \end{align*}
    a similar argument shows that
    \begin{equation*}
        K_{\omega,c}\sts{\psi}\geqslant K_{\omega,c}\sts{\psi^{\ast}}.
    \end{equation*}
    Since $\displaystyle \lim_{n\to\infty}H_{\omega,c}\sts{\varphi_n}= J^0_{\omega,c},$ we have $\varphi_n$ is bounded in $\th,$
    which means $\phi_n$ is bounded in $H^1\sts{\R,\R}.$ It follows from Lemma \ref{lem:compact} that
     there exists $\phi\in H^1\sts{\R,\R}$ such that
     \begin{align*}
        \lim_{n\to\infty}\phi_n &= \phi,\quad\text{weakly in } H^1\sts{\R,\R},
        \\
        \lim_{n\to\infty}\phi_n &= \phi,\quad\text{strongly in } L^6\sts{\R,\R},
        \\
        \lim_{n\to\infty}\phi_n &= \phi,\quad\text{strongly in } L^4\sts{\R,\R}.
    \end{align*}
     Hence, from the definition of $\th,$
          \begin{align*}
        \lim_{n\to\infty}\varphi_n = \varphi,\quad &\text{weakly in } \th,
        \\
        \lim_{n\to\infty}\varphi_n = \varphi,\quad &\text{strongly in } L^6\sts{\R,\C}.
        \\
        \lim_{n\to\infty}\varphi_n = \varphi,\quad &\text{strongly in } L^4\sts{\R,\C},
        \end{align*}
    where $\varphi=e^{i\frac{c}{2} x}\phi.$
     It follows from the weak lower continuity of the norm that
     \begin{align*}
       H_{\omega,c}\sts{\varphi} &\leqslant \lim_{n\to\infty}H_{\omega,c}\sts{\varphi_n}=J^0_{\omega,c} \\
       K_{\omega,c}\sts{\varphi} &\leqslant \liminf_{n\to\infty}K_{\omega,c}\sts{\varphi_n}\leqslant 0.
     \end{align*}

     Next, we shall prove $\varphi\neq 0.$ Suppose that $\varphi=0,$
     then we have
     \begin{align*}
      0\leq \liminf_{n\to\infty}K^Q_{\omega,c}\sts{\varphi_n}
       &=
        \liminf_{n\to\infty}\sts{ K_{\omega,c}\sts{\varphi_n} +K^N_{\omega,c}\sts{\varphi_n} }
        \\
        &\leqslant
            \liminf_{n\to\infty}K_{\omega,c}\sts{\varphi_n} + \lim_{n\to\infty}K^N_{\omega,c}\sts{\varphi_n}\leqslant 0.
     \end{align*}
By Lemma \ref{Lem:knear0}, there exists a subsequence
$\varphi_{n_k}$ such that
     \begin{equation*}
       K_{\omega,c}\sts{\varphi_{n_k}}>0, \quad \text{for $k$ large enough,}
     \end{equation*}
It is a contradiction with the choice of $\varphi_n$. Thus
$\varphi\neq 0,$ Hence $\varphi$ is a minimizer of
\eqref{Sub:mini:ineq:struct}. By Lemma \ref{lem:sub:equiminimization}, $\varphi$
is also a minimizer of \eqref{Sub:mini:eq:struct}.
\end{proof}
Since $J_{\omega,c}$ and $K_{\omega,c}$ are $C^1$ functionals on $\th,$ by the above lemma,
it is easy to see that if $\varphi\in \th\setminus\ltl{0}$ is a
minimizer for \eqref{Sub:mini:eq:struct}, then there exists $\eta\in\R$ such
that
\begin{equation*}
  \dual{J_{\omega,c}'\sts{\varphi}}{\psi} =\eta \dual{K_{\omega,c}'\sts{\varphi}}{\psi},\quad\text{for any~~}\psi\in \th,
\end{equation*}
specially, if we take $\psi=\varphi$ in the above equation, then it
follows from \eqref{eq:varphi:II:invariant} that
\begin{align*}
  0=K_{\omega,c}\sts{\varphi} &=   \int\sts{\abs{\partial_x\varphi}^2-\frac{3}{16}\abs{\varphi}^6 +\omega
  \abs{\varphi}^2
  -\frac{c}{2}{\Im\sts{\overline{\varphi}\varphi_x}+\frac{c}{2}\abs{\varphi}^4
  }}
  \\
  &=\eta \int\sts{2\abs{\partial_x\varphi}^2 -\frac{9}{8}\abs{\varphi}^6+2\omega\abs{\varphi}^2- 2c\Im\sts{\overline{\varphi}\varphi_x}
  +2c\abs{\varphi}^4}
  \\
  &=
    4\eta K_{\omega,c}\sts{\varphi}-2\eta\int\sts{\abs{\partial_x\varphi}^2 +\omega\abs{\varphi}^2- c\Im\sts{\overline{\varphi}\partial_x\varphi}} -\frac{3}{8}\eta \int\abs{\varphi}^6
  \\
  &= - 2\eta\normth{\varphi}^2 -\frac{3}{8}\eta \normsix{\varphi}^4.
\end{align*}
Since $\varphi\in\th\setminus\ltl{0},$ we obtain $\eta=0$ and
\begin{equation*}
  J_{\omega,c}'\sts{\varphi}=0,\quad \text{in~~} \th^*,
\end{equation*}
i.e. $\varphi$ satisfies \eqref{Eq:varphi:II} in sense of $\th.$
Since $\varphi(x)=e^{i\frac{c}{2} x}\phi(x)$, we have
\begin{equation}\label{Sub:EII:H1}
  \sts{\omega-\frac{c^2}{4}}\phi-\partial^2_x \phi
    -\frac{3}{16}\abs{\phi}^4\phi
  = -\frac{c}{2}\abs{\phi}^2\phi,\quad\text{in~~} H^1\sts{\R, \C}.
\end{equation}

Note that $\phi_{\omega, c}$ in \eqref{solt:sub} is a solution to
\eqref{Sub:EII:H1}. By the uniqueness result (Theorem 8.1.6 in
\cite{Caz:NLS:book}, ODE argument), we have
\begin{proposition}\label{exist:subcritical case}
For subcritical case $4\omega>c^2, $ up to the phase rotation and
spatial translation symmetries, \eqref{Eq:varphi:I} has a unique
solution $\varphi_{\omega,c}$ in $H^1(\R, \C)$, where
    \begin{equation*}
      \varphi_{\omega,c}\sts{x}=e^{\i\frac{c}{2} x}\left[\frac{\sqrt{\omega}}{4\omega-c^2}
\left\{\cosh\big(\sqrt{4\omega-c^2}x\big)-\frac{c}{2\sqrt{\omega}}\right\}\right]^{-1/2}.
    \end{equation*}
\end{proposition}

\subsection{Variational characterization for the critical case $4\omega=c^2, c>0$} \label{sect:vc:critical}
For the critical parameters $4\omega=c^2, c>0$,  the quadratic terms
of the functionals $J_{\omega,c}\sts{\varphi}$ and
$K_{\omega,c}\sts{\varphi}$ do not enjoy coercivity in $H^1$, hence
we can not preform the variational method (minimization) in the
framework of \cite{ColinOhta-DNLS} directly. Here we combine the
variational method with the structure analysis to show
the existence of the solitary waves to \eqref{Eq:varphi:II}. The similar structure analysis also occurs in \cite{BethuelGS-GP-II, Gerard:GP,
Maris:NLS}. We first solve the minimization problem in the weak space $X_c$ with
structure, then show the uniqueness\footnote{Up to the phase
rotation and spatial translation symmetries.} and the $H^1$
regularity of the solitary waves. Therefore we can solve  the
minimization problem in the energy space.

Based on the structure analysis in Section \ref{sect:structure:compactness}, we now consider the following space
   \begin{equation}\label{Def:space_crit}
     X_c:=\ltl{ \varphi\in\mathcal{S}'
        ~:~
        \varphi\sts{x}=e^{\i\frac{c}{2} x}\phi\sts{x}
        \,\, \text{with}\,\, \phi\in \sts{\dot{H}^1\cap L^4}\sts{\R,\C}}
   \end{equation}
with the norm
       \begin{equation*}
        \normx{\varphi}:=\normdh{\phi}+\normf{\phi}\quad\text{with } \phi\in \sts{\dot{H}^1\cap L^4}\sts{\R,\C}.
    \end{equation*}
   It is clear that $\sts{X_c,\normx{\cdot}}$ is a Banach space and $H^1\sts{\R, \C}\hookrightarrow X_c$.

First, we consider the functional $J_{\omega,c}$ on  $X_c$ instead
of $H^1.$
   Similarly, it is easy to check that $J_{\omega,c}$ is (at least) a $C^2$ functional and unbounded from below on
   $X_c$.
Moreover, $\varphi\in X_c\setminus\ltl{0}$ is a solution of
\eqref{Eq:varphi:II} if and only if $\varphi\in X_c\setminus\ltl{0}$
is a critical point of the functional $J_{\omega,c}$.


Similarly to the subcritical case, we consider the following
minimization problem
\begin{equation}\label{Crit:mini:eq:struct}
J^0_{\omega,c}=\inf
\ltl{J_{\omega,c}\sts{\varphi}~:~K_{\omega,c}\sts{\varphi}=0,~~\varphi\in
X_c\setminus\ltl{0}},
\end{equation}
and  define
\begin{align*}
  K^{Q}_{\omega,c}\sts{\varphi} &:= \int\sts{\abs{\partial_x\varphi}^2 +\omega\abs{\varphi}^2
  -c\Im\sts{\overline{\varphi}\partial_x\varphi}
  +\frac{c}{2}\abs{\varphi}^4
  }dx
  \\
  &=\int\sts{ \abs{ \partial_x\sts{ e^{-\i \frac{c}{2}x }\varphi }}^2 +\frac{c}{2}\abs{\varphi}^4
  }dx \geq 0,
  \\
  K^{N}_{\omega,c}\sts{\varphi} &:=K^Q_{\omega,c}\sts{\varphi}-K_{\omega,c}\sts{\varphi}=\frac{3}{16}\int\abs{\varphi}^6 dx.
\end{align*}
By the definition,  we have for $\lambda>0$ \[
K^Q_{\omega,c}\sts{\lambda\varphi} = \lambda^2
\int\sts{\abs{\partial_x\varphi}^2 +\omega\abs{\varphi}^2
  -c\Im\sts{\overline{\varphi}\partial_x\varphi}}dx+
\lambda^4\int \abs{\varphi}^4dx .\] This implies that
\begin{lemma}\label{Lem:crit:smallscal}For any $\varphi\in X_c \setminus\ltl{0}$, we
have \[ \lim_{\lambda \rightarrow 0 +}
K^Q_{\omega,c}\sts{\lambda\varphi} =0.\]
\end{lemma}

The next lemma exhibits the behavior of $K_{\omega,c}$ near the origin of $X_c.$
\begin{lemma}
\label{Lem:crit:knear0}
  For any bounded sequence $\ltl{\varphi_n}\subset X_c\setminus\ltl{0}$ with
  \begin{equation*}
    \lim_{n\to\infty}K^Q_{\omega,c}\sts{\varphi_n}=0.
  \end{equation*}
  We have for large $n,$
  \begin{equation*}
    K_{\omega,c}\sts{\varphi_n}>0.
  \end{equation*}
\end{lemma}
\begin{proof}
  Since $\varphi_n\in X_c\setminus\ltl{0}$, there exists $\phi_n\in \sts{\dot{H}^1\cap L^4}\sts{\R,\C}$, such that
  $\varphi_n\sts{x}=e^{\i \frac{c}{2}x}\phi_n\sts{x}$ with
   \begin{align*}
     K^Q_{\omega,c}\sts{\varphi_n} &= \int\sts{\abs{\partial_x\varphi_n}^2+
     \omega\abs{\varphi_n}^2 - c\Im\sts{\overline{\varphi}_n \partial_x\varphi_n}+\frac{c}{4}\abs{\varphi_n}^4 }dx \\
     & =\int\sts{ \abs{\partial_x
     \sts{e^{-\i\frac{c}{2}x}\varphi_n}}^2+\frac{c}{4}\abs{e^{-\i\frac{c}{2}x}\varphi_n}^4}dx
     \\
     &= \int \sts{\abs{\partial_x\phi_n}^2+\frac{c}{4}\abs{\phi_n}^4 }dx.
   \end{align*}

   By $\lim_{n\to\infty}K^Q_{\omega,c}\sts{\varphi_n}=0,$ it follows from the Gagliardo-Nirenberg and H\"{o}lder inequalities that
  \begin{align*}
    K^{N}_{\omega,c}\sts{\varphi_n} \approx \int\abs{\phi_n}^6
    \lesssim \big\|\phi_n\big\|^{2/3}_{\dot
    H^1}\normf{\phi_n}^{16/3}
  \lesssim \normdh{\phi_n}^4+\normf{\phi_n}^{32/5}
       =\o{ K^Q_{\omega,c}\sts{\varphi_n} }.
  \end{align*}
  Thus, for sufficiently large $n$, we have
  \begin{equation*}
    K_{\omega,c}\sts{\varphi_n}=K^Q_{\omega,c}\sts{\varphi_n}-K^{N}_{\omega,c}\sts{\varphi_n}\approx K^Q_{\omega,c}\sts{\varphi_n}>0.
  \end{equation*}
  This completes the proof.
\end{proof}
%
We now replace the functional $J_{\omega,c} $ in \eqref{Crit:mini:eq:struct}, which is unbounded from below,
 with a positive functional
$H_{\omega,c} $, while extending the minimizing region from
``$K_{\omega,c} =0$'' to ``$K_{\omega,c} \leq 0$''. Let
\begin{align}\label{Funct:H:crit}
  H_{\omega,c}\sts{\varphi} := & \; J_{\omega,c}\sts{\varphi}-\frac{1}{6}K_{\omega,c}\sts{\varphi}
  \nonumber \\
  =&\;\frac{1}{3} \int\sts{\abs{\varphi_x}^2 +\omega\abs{\varphi}^2-c\Im\sts{\overline{\varphi}\partial_x\varphi}
  +\frac{c}{8}\abs{\varphi}^4
  }dx \\
  \geq &\; 0.
\end{align}
In addition, for any $\varphi\in X_c\setminus\ltl{0} $ and
$0<\lambda_1 < \lambda_2$, we have the following monotonicity.
\begin{align}\label{H:crit:prop}
  H_{\omega,c} \sts{\lambda_1\varphi}<  H_{\omega,c}\sts{\lambda_2\varphi}.
\end{align}
 In order to find the minimizers of \eqref{Crit:mini:eq:struct}, we shall
consider the following constrained minimization problem
\begin{equation}\label{Crit:mini:ineq:struct}
  \tilde{J}^0_{\omega,c}=\inf\ltl{H_{\omega,c}\sts{\varphi}~:~K_{\omega,c}\sts{\varphi}\leqslant 0,~~\varphi\in X_c\setminus\ltl{0} },
\end{equation}
The following lemma shows that two minimization problems
\eqref{Crit:mini:eq:struct} and \eqref{Crit:mini:ineq:struct} are equivalent.
\begin{lemma}
\label{lem:crit:equiminimization}
  Let $J^0_{\omega,c}$ and $\tilde{J}^0_{\omega,c}$ be defined by \eqref{Crit:mini:eq:struct} and \eqref{Crit:mini:ineq:struct} respectively.
  Then we have  \begin{enumerate}
\item   $J^0_{\omega,c}=\tilde{J}^0_{\omega,c}>0.$
\item  any minimizer for \eqref{Crit:mini:eq:struct}  is also a minimizer for \eqref{Crit:mini:ineq:struct}, and vice versa.
\end{enumerate}
\end{lemma}
\begin{proof}
  First, by definition, we have
  \begin{equation*}
    H_{\omega,c}\sts{\varphi}=J_{\omega,c}\sts{\varphi} \quad\text{for any } \varphi\in X_c\setminus\ltl{0} \quad\text{ with } K_{\omega,c}\sts{\varphi}=0,
  \end{equation*}
and
  \begin{align*}
    J^0_{\omega,c} &= \inf\ltl{J_{\omega,c}\sts{\varphi}~:~K_{\omega,c}\sts{\varphi}=0,~~\varphi\in X_c\setminus\ltl{0}} \\
    &= \inf\ltl{H_{\omega,c}\sts{\varphi}~:~K_{\omega,c}\sts{\varphi}=0,~~\varphi\in X_c\setminus\ltl{0}}
    \\
    & \geqslant \inf\ltl{H_{\omega,c}\sts{\varphi}~:~K_{\omega,c}\sts{\varphi}\leqslant 0,~~\varphi\in X_c\setminus\ltl{0}}\\
    & =\tilde{J}^0_{\omega,c}.
  \end{align*}
Next, for any $\varphi\in X_c\setminus\ltl{0}$ with
$K_{\omega,c}\sts{\varphi}<0$. By Lemma \ref{Lem:crit:smallscal} and Lemma
\ref{Lem:crit:knear0}, there exists $\lambda_0\in\sts{0,1}$ such that
$K_{\omega,c}\sts{\lambda_0 \varphi}=0.$  The monotonicity \eqref{H:crit:prop} of
the functional $H_{\omega,c}$ implies that
  \begin{equation*}
    J_{\omega,c}\sts{\lambda_0 \varphi} = H_{\omega,c}\sts{\lambda_0 \varphi}<H_{\omega,c}\sts{\varphi}.
  \end{equation*}
  Hence, we have $J^0_{\omega,c}\leqslant \tilde{J}^0_{\omega,c}$, which implies $(1)$.

  Next, we show $(2)$. On one hand, let $\varphi$ be any minimizer for $\tilde{J}^0_{\omega,c},$ i.e.
    \begin{equation*}
        \varphi\in X_c\setminus\ltl{0} \text{ with } K_{\omega,c}\sts{\varphi}\leqslant 0\text{ and } H_{\omega,c}\sts{\varphi}=\tilde{J}^0_{\omega,c}.
    \end{equation*}
    In order to show that $\varphi$ is also a minimizer for $J^0_{\omega,c}$, we only need to show that
    $K_{\omega,c}\sts{\varphi}=0.$
 We argue by contradiction. Assume that $K_{\omega,c}\left( \varphi\right)<0,$
 by Lemma \ref{Lem:crit:smallscal} and Lemma \ref{Lem:crit:knear0}, there exists $\lambda_0\in\sts{0,1}$ which is dependent on $\varphi$ such that
    \begin{equation*}
      K_{\omega,c}\sts{\lambda_0 \varphi}=0
    \end{equation*}
    and
    \begin{equation*}
      K_{\omega,c}\sts{\lambda \varphi}<0,\quad \text{for any } \lambda\in\left(\lambda_0,1\right].
    \end{equation*}
Thus  by the monotonicity \eqref{H:crit:prop} of the functional $H$,  we
obtain that
\begin{equation}
  \label{eq:crit:minimizer-value-lem}
  \tilde{J}^0_{\omega,c}=H_{\omega,c}\left( \varphi \right)>H_{\omega,c}\left( \lambda_0 \varphi \right)=J_{\omega,c}\left( \lambda_0\varphi \right)\geqslant J^0_{\omega,c}=\tilde{J}^0_{\omega,c},
\end{equation}
which is a contradiction. Hence,
\begin{math}
  K_{\omega,c}\left( \varphi \right)=0
\end{math}
and
\begin{math}
   \varphi
\end{math}
is also a minimizer for $J^0_{\omega,c}$. On the other hand, let
\begin{math}
   \varphi
\end{math}
be any minimizer for $J^0_{\omega,c},$ i.e.
\begin{equation*}
    \varphi\in X_c\setminus\ltl{0} \text{ with } K_{\omega,c}\sts{\varphi}=0, \text{ and } J_{\omega,c}\sts{\varphi}=J^0_{\omega,c}.
\end{equation*}
Then we have
\begin{equation*}
  \tilde{J}^0_{\omega,c}\leqslant H_{\omega,c}\left( \varphi \right)=J_{\omega,c}\left( \varphi \right)=J^0_{\omega,c}=\tilde{J}^0_{\omega,c}.
\end{equation*}
Hence,
\begin{math}
  \varphi
\end{math}
is also a minimizer for $\tilde{J}^0_{\omega,c}.$ This completes the proof.
\end{proof}

Now, we can use the non-increasing rearrangement technique in
\cite{LiebL:book} once again to show the existence of minimizer to
\eqref{Crit:mini:eq:struct}.

\begin{lemma}
\label{Lem:crit:exist} There exists at least one minimizer for the
minimization problem \eqref{Crit:mini:eq:struct}.
\end{lemma}
\begin{proof}
   Let $\ltl{\varphi_n}\subset X_c\setminus\ltl{0}$ be a minimizing sequence of the constrained problem \eqref{Crit:mini:ineq:struct}, i.e.
  \begin{align*}
    K_{\omega,c}\sts{\varphi_n}\leqslant 0,\quad H_{\omega,c}\sts{\varphi_n}\geqslant J^0_{\omega,c} \text{~~and~~} \lim_{n\to\infty}H\sts{\varphi_n}= J^0_{\omega,c}.
  \end{align*}
By definition of $X_c,$ there exists a sequence $\ltl{\phi_n}\in
\sts{\dot{H}^1\cap L^4}\sts{\R,\C}\setminus\ltl{0}$ such that
    \begin{equation*}
~~\varphi_n=e^{\i \frac{c}{2}x }\phi_n~~\text{and~~}
\normx{\varphi_n}=\normdh{\phi_n}+\normf{\phi_n}.
    \end{equation*}
    Without loss of generality, we may also assume that $\phi_n$ are the real valued, radially symmetric and non-increasing functions about the origin of $\R$.
    Indeed, for any $\psi\in X_c\setminus\ltl{0}$ with $\psi=e^{\i x}\phi$ and $\normx{\psi}=\normdh{\phi}+\normf{\phi},$
    let $\phi^{\ast}$ be the Schwarz symmetrization of $\phi,$ and $\psi^{\ast}= e^{\i \frac{c}{2}x }\phi^{\ast},$
    by Schwarz rearrangement inequality in \cite{LiebL:book}, it is easy to check that
    \begin{align*}
        H_{\omega,c}\sts{\psi}
        & =\frac{1}{3}\int\sts{\abs{\partial_x\varphi}^2 +\omega\abs{\varphi}^2-c\Im\sts{\overline{\varphi}\partial_x\varphi}+\frac{c}{8}\abs{\psi}^4 }dx\\
        &= \frac{1}{3}\int \sts{\abs{\sts{e^{-\i \frac{c}{2}x } \psi}_x}^2 + \frac{c}{8}\abs{e^{-\i \frac{c}{2}x } \psi}^4 }dx\\
        &=\frac{1}{3}\int\sts{\abs{\partial_x\phi}^2 + \frac{c}{8}\abs{\phi}^4
        }dx
        \\
        &\geqslant \frac{1}{3}\int\sts{\abs{\partial_x\phi^{\ast}}^2 + \frac{c}{8}\abs{\phi^{\ast}}^4
        }dx\\
        & =  H_{\omega,c}\sts{\psi^{\ast}},
    \end{align*}
    a similar argument shows that
    \begin{equation*}
        K_{\omega,c}\sts{\psi}\geqslant K_{\omega,c}\sts{\psi^{\ast}}.
    \end{equation*}
    Since $\displaystyle \lim_{n\to\infty}H_{\omega,c}\sts{\varphi_n}= J^0_{\omega,c},$ we have $\varphi_n$ is bounded in $X_c,$
    which means $\phi_n$ is bounded in $\dot{H}^1\sts{\R,\R}\cap L^4\sts{\R,\R}.$ It follows from Lemma \ref{lem:compact} that
     there exists $\phi\in \dot{H}^1\sts{\R,\R}\cap L^4\sts{\R,\R}$ such that
     \begin{align*}
        \lim_{n\to\infty}\phi_n &= \phi,\quad\text{weakly in } \dot{H}^1\sts{\R,\R}\cap L^4\sts{\R,\R},
        \\
        \lim_{n\to\infty}\phi_n &= \phi,\quad\text{strongly in } L^6\sts{\R,\R}.
    \end{align*}
     Hence, from the definition of $X_c,$
          \begin{align*}
        \lim_{n\to\infty}\varphi_n &= \varphi=e^{\i\frac{c}{2}x}\phi,\quad\text{weakly in } X_c,
        \\
        \lim_{n\to\infty}\varphi_n &= \varphi,\quad\quad \text{strongly in } L^6\sts{\R,\C}.
    \end{align*}
     It follows from the weak lower continuity of the norm that
     \begin{align*}
       H_{\omega,c}\sts{\varphi} &\leqslant \lim_{n\to\infty}H_{\omega,c}\sts{\varphi_n}=J^0_{\omega,c}, \\
       K_{\omega,c}\sts{\varphi} &\leqslant \liminf_{n\to\infty}K_{\omega,c}\sts{\varphi_n}\leqslant 0.
     \end{align*}

     Next, we shall prove $\varphi\neq 0.$ Suppose that $\varphi=0,$
     then we have
     \begin{align*}
      0\leq \liminf_{n\to\infty}K^Q_{\omega,c}\sts{\varphi_n}
       &=
        \liminf_{n\to\infty}\sts{ K_{\omega,c}\sts{\varphi_n} +K^{N}_{\omega,c}\sts{\varphi_n} }
        \\
        &\leqslant
            \liminf_{n\to\infty}K_{\omega,c}\sts{\varphi_n} + \lim_{n\to\infty}K^{N}_{\omega,c}\sts{\varphi_n}\\
        & \leqslant 0.
     \end{align*}
By Lemma \ref{Lem:crit:knear0}, there exists a subsequence
$\varphi_{n_k}$ such that
     \begin{equation*}
       K_{\omega,c}\sts{\varphi_{n_k}}>0, \quad \text{for $k$ large enough,}
     \end{equation*}
It is a contradiction with the choice of $\varphi_n$. Thus
$\varphi\neq 0,$ Hence $\varphi$ is a minimizer of
\eqref{Crit:mini:ineq:struct}. By Lemma \ref{lem:crit:equiminimization}, $\varphi$
is also a minimizer of \eqref{Crit:mini:eq:struct}.
\end{proof}
Since $J_{\omega,c}$ and $K_{\omega,c}$ are $C^1$ functionals on $X_c,$ by the above lemma,
it is easy to see that if $\varphi\in X_c\setminus\ltl{0}$ is a
minimizer for \eqref{Crit:mini:eq:struct}, then there exists $\eta\in\R$ such
that
\begin{equation*}
  \dual{J_{\omega,c}'\sts{\varphi}}{\psi} =\eta \dual{K_{\omega,c}'\sts{\varphi}}{\psi},\quad\text{for any~~}\psi\in
  X_c.
\end{equation*}
If we take $\psi=\varphi$ in the above equation, then it follows
from \eqref{eq:varphi:II:invariant} that
\begin{align*}
  0&=K_{\omega,c}\sts{\varphi} \\
  &
  =\int\sts{\abs{\partial_x\varphi}^2-\frac{3}{16}\abs{\varphi}^6+\omega\abs{\varphi}^2
  -c\Im\sts{\overline{\varphi}\partial_x\varphi}+\frac{c}{2}\abs{\varphi}^4
  }dx
  \\
  &=\eta \int\sts{2\abs{\partial_x\varphi}^2 -\frac{9}{8}\abs{\varphi}^6+2\omega\abs{\varphi}^2- 2c\Im\sts{\overline{\varphi}\partial_x\varphi}
  +2c\abs{\varphi}^4}dx
  \\
  &=6\eta K_{\omega,c}\sts{\varphi}-4\eta\int\sts{\abs{\partial_x\varphi}^2 +\frac{c^2}{4}\abs{\varphi}^2- c
  \Im\sts{\overline{\varphi}\partial_x\varphi}}dx
  -c\eta\int\abs{\varphi}^4dx
  \\
  &= - 4\eta\|\phi\|^2_{\dot H^1} -c\eta \normf{\phi}^4,
\end{align*}
where $\varphi\sts{x}=e^{\i \frac{c}{2}x}\phi\sts{x}.$ Since $\varphi\in
X_c\setminus\ltl{0},$ we obtain $\eta=0$ and
\begin{equation*}
  J_{\omega,c}'\sts{\varphi}=0,\quad \text{in~~} X_c^*,
\end{equation*}
i.e. $\varphi$ satisfies \eqref{V:EI-b} in sense of $X_c$. Since
$\varphi(x)=e^{\i \frac{c}{2}x }\phi(x)$, we have
\begin{equation}\label{EII:crit:H1L4}
  -\partial^2_x\phi+\frac{c}{2}\abs{\phi}^2\phi-\frac{3}{16}\abs{\phi}^4\phi=0 \quad\text{in~~} \dot{H}^1\sts{\R,\C}\cap L^4\sts{\R,\C}.
\end{equation}
On the other hand, by the sharp Gagliardo-Nirenberg inequality in
\cite{Agueh}, \eqref{EII:crit:H1L4} has a radial symmetric solution
$\phi_{\omega, c}\sts{x}=\frac{2\sqrt{c}}{\sqrt{c^2x^2+1}}$. In
addtion, by the similar uniqueness result as Theorem 8.1.6 in
\cite{Caz:NLS:book}(ODE argument), it is unique, up to the phase
rotation and spatial translation symmetries. Last it is easy to
verify that
\begin{equation}\label{V:l2reg}
  \phi_{\omega,c}\in H^1\sts{\R,\C}.
\end{equation}
It follows that
\begin{align*}
  \left\|\varphi_{\omega,c}\right\|_{H^1}^2=&\left\|e^{-\i \frac{c}{2}x }\phi_{\omega,c}\right\|_{H^1}^2  \lesssim \left\|\phi_{\omega,c}\right\|_{H^1}^2 + \left\|\phi_{\omega,c}\right\|_{L^2}^2<\infty,
\end{align*}
which means $\varphi_{\omega,c}\in H^1\sts{\R, \C}\hookrightarrow
X_c.$ Thus, we have
\begin{align}
J^0_{\omega,c}=&
J_{\omega,c}\sts{\varphi_{\omega,c}} \notag \\
=&\inf\ltl{J_{\omega,c}\sts{\varphi}~:~K_{\omega,c}\sts{\varphi}=0,~~\varphi\in
H^1\sts{\R, \C}\setminus\ltl{0}}. \label{crit:mini:energy space}
\end{align}

Summing up,  we have
\begin{proposition}\label{exist:critical case}
For the critical case $4\omega=c^2, c>0$, up to the phase rotation
and spatial translation symmetries, \eqref{Eq:varphi:I} has a unique
solution $\varphi_{\omega,c}$ in $H^1\sts{\R, \C}$, where
    \begin{equation*}
      \varphi_{\omega,c}\sts{x}=e^{\i \frac{c}{2}x}\frac{2\sqrt{c}}{\sqrt{c^2x^2+1}}.
    \end{equation*}
\end{proposition}

\begin{proof}[Proof of Theorem \ref{Thm:Existence of TW}]
the existence and uniqueness of the solitary waves in the energy space
are obtained by Proposition \ref{exist:subcritical case} and
Proposition \ref{exist:critical case}, while the nonexistence of the
solitary waves in the energy space is obtain by Proposition
\ref{nonexist:crit case}.
\end{proof}
%
%
%
%

\section{Global well-posedness result for solutions with initial data in $\KKK^+_{\omega,c}$}\label{sect:GWP}
In this section, we show Theorem \ref{Thm:GWP} and Corollary \ref{cor:gwp}. In order to do this,
we first show the following uniformly boundedness of $K_{\omega, c}$
functional in the energy space.
\begin{lemma}\label{lem:Kunibound}
  Assume $(\omega,
c)$ with $4\omega>c^2$ or $4\omega=c^2, c>0$ and let $\varphi\in
\KKK^+_{\omega,c}$, then we have
  \begin{align*}
     K_{\omega,c}\sts{\varphi}\geqslant\min\ltl{~4\sts{ J^0_{\omega,c}-J_{\omega,c}\sts{\varphi}},~\frac14\normdh{e^{-i\frac{c}{2}x} \varphi}^2+\frac14\sts{\omega-\frac{c^2}{4}}\normto{\varphi}^2~}.
  \end{align*}
\end{lemma}
\begin{proof}
  For the simply of notation, for any $\varphi\in \KKK^+_{\omega,c},$ we denote
  \begin{align*}
    \mathfrak{j}\sts{\lambda}:=J_{\omega,c}\sts{e^{\lambda}\varphi },
  \end{align*}
then, it is easy to see that
\begin{align*}
  \lim_{\lambda\to -\infty}\mathfrak{j}\sts{\lambda} = 0,\quad
  \mathfrak{j}'\sts{\lambda}= K_{\omega,c}\sts{e^{\lambda}\varphi },
\end{align*}
and
\begin{align}
  & \mathfrak{j}''\sts{\lambda} \notag\\
  =& \int\sts{2e^{2\lambda}\abs{\partial_x\varphi}^2 -2c e^{2\lambda}\Im\sts{\overline{\varphi}\partial_x\varphi}
  +2\omega e^{2\lambda}\abs{\varphi}^2-\frac{9}{8}e^{6\lambda}\abs{\varphi}^6  +2 c e^{4\lambda}\abs{\varphi}^4
  }dx \notag\\
  =& 4\int\sts{e^{2\lambda}\abs{\partial_x\varphi}^2
  -c e^{2\lambda}c\Im\sts{\overline{\varphi}\partial_x\varphi}
  +\omega e^{2\lambda}\abs{\varphi}^2 -\frac{3}{16}e^{6\lambda}\abs{\varphi}^6
  +\frac{c}{2}e^{4\lambda}\abs{\varphi}^4}dx \notag
  \\
  &-\frac{3}{8}\int\sts{e^{6\lambda}\abs{\varphi}^6}-2 e^{2\lambda}\int\sts{\abs{\partial_x\varphi}^2
  -c\Im\sts{\overline{\varphi}\partial_x\varphi}
  +\omega \abs{\varphi}^2 }dx \notag
  \\
  \leqslant & 4\mathfrak{j}'\sts{\lambda}-2e^{2\lambda}\sts{\normdh{e^{-i\frac{c}{2}x} \varphi}^2+\sts{\omega-\frac{c^2}{4}}\normto{\varphi}^2}.\label{jineq}
\end{align}

We will discuss in two cases:

\noindent Case (a): $8 K_{\omega,c}\sts{\varphi }\geqslant
2\sts{\normdh{e^{-i\frac{c}{2}x}
\varphi}^2+\sts{\omega-\frac{c^2}{4}}\normto{\varphi}^2}.$ Then, we
have
      \begin{align*}
        K_{\omega,c}\sts{\varphi }\geqslant \frac14\normdh{e^{-i\frac{c}{2}x} \varphi}^2+\frac14 \sts{\omega-\frac{c^2}{4}}\normto{\varphi}^2.
      \end{align*}

\noindent Case (b): $8 K_{\omega,c}\sts{\varphi
}<2\sts{\normdh{e^{-i\frac{c}{2}x}
\varphi}^2+\sts{\omega-\frac{c^2}{4}}\normto{\varphi}^2}.$ By
\eqref{jineq}, we have for $\lambda=0,$
      \begin{align}\label{jcase2ineq}
       \nonumber 0\leqslant& 8 \mathfrak{j}'\sts{\lambda}< 2e^{2\lambda}\sts{\normdh{e^{-i\frac{c}{2}x} \varphi}^2+\sts{\omega-\frac{c^2}{4}}\normto{\varphi}^2},
        \\
       \nonumber \mathfrak{j}''\sts{\lambda}\leqslant & 4\mathfrak{j}'\sts{\lambda}-2e^{2\lambda}\sts{\normdh{e^{-i\frac{c}{2}x} \varphi}^2+\sts{\omega-\frac{c^2}{4}}\normto{\varphi}^2}
        \\
        \leqslant & -4\mathfrak{j}'\sts{\lambda}.
      \end{align}
It follows from the continuity of $\mathfrak{j}'$ and
$\mathfrak{j}''$ with respect to $\lambda$ that, $\mathfrak{j}'$ is
decreasing as $\lambda$ increases until
$\mathfrak{j}'\sts{\lambda_0}=0$ for some finite $\lambda_0>0.$
Moreover, \eqref{jcase2ineq} holds on $[0,\lambda_0].$ By $
K_{\omega,c}\sts{e^{\lambda_0}\varphi } =0,$ we have
      \begin{align*}
        J_{\omega,c}\sts{e^{\lambda_0}\varphi }\geqslant J_{\omega,c}^0.
      \end{align*}
      Now, integrating the second inequality in \eqref{jcase2ineq}, we obtain
      \begin{align*}
        -K_{\omega,c}\sts{ \varphi }= \mathfrak{j}'\sts{\lambda_0}-\mathfrak{j}'\sts{0} \leqslant -4\sts{ \mathfrak{j}\sts{\lambda_0}-\mathfrak{j}\sts{0} }\leqslant -4 \sts{ J_{\omega,c}^0-J_{\omega,c}\sts{\varphi }
        }.
      \end{align*}
This ends the proof.
\end{proof}

\begin{proof}[Proof of Theorem \ref{Thm:GWP}]
We first show $(1)$. It suffices to deal with the subcritical case
$(4\omega>c^2)$, since the critical case $(4\omega=c^2, c>0)$ can be
handled in the same way.

First of all, we define the function $j ~:~[0,\infty)\mapsto\R,$
\begin{align*}
  j\sts{\lambda}:= & J_{\omega,c}\sts{\lambda\varphi_{\omega,c}},
\end{align*}
where $\varphi_{\omega,c}$ is the minimizer obtained by Lemma
\ref{Lem:sub:exist}.
On one hand, it is easy to see that
  \begin{align*}
    \lim_{\lambda\to 0+} j\sts{\lambda}=&\; 0.
  \end{align*}
Moreover, it follows from Lemma \ref{lem:sub:equiminimization} that
\begin{align*}
j\sts{1}=J^0_{\omega,c}>0.
\end{align*}
Thus,
\begin{align*}
  j\sts{\lambda}<j\sts{1}=J^0_{\omega,c},\quad \text{for~~} \lambda \text{~~close to zero,}
\end{align*}
which means
\begin{align}\label{GWP:J:small}
  J_{\omega,c}\sts{\lambda\varphi_{\omega,c}}<J^0_{\omega,c},\quad \text{for~~} \lambda \text{~~close to zero.}
\end{align}
On the other hand, it follows from Lemma \ref{Lem:sub:smallscal} and
Lemma \ref{Lem:knear0} that
\begin{equation}\label{GWP:K:ge0}
  K_{\omega,c}\sts{\lambda\varphi_{\omega,c}}>0\quad \text{for~~} \lambda \text{~~close to zero.}
\end{equation}
\eqref{GWP:J:small} and \eqref{GWP:K:ge0} imply that
\begin{equation*}
  \lambda\varphi_{\omega,c}\in \KKK^+_{\omega,c}\quad \text{for~~} \lambda \text{~~close to zero,}
\end{equation*}
i.e. $\KKK^+_{\omega,c}\neq\emptyset.$ In the same way, by taking
$\lambda$ large enough, one can show that
$\KKK^-_{\omega,c}\neq\emptyset.$ By the variational argument of the
solitary waves in Section \ref{sect:exit} and the standard argument, we
know the invariance of the sets $\KKK^+_{\omega,c}$ and
$\KKK^-_{\omega,c}$ under the flow \eqref{DNLS}.

Next, we show $(2).$ We define the set
\begin{align*}
  \Omega_{\omega,c}^{+}=\ltl{ t\in I~:~J_{\omega,c}\sts{u\sts{t}}<J_{\omega,c}^0,~~K_{\omega,c}\sts{u\sts{t}}\geqslant 0}.
\end{align*}
First, by the assumption $u_0\in \KKK^+_{\omega,c},$ we have $0\in
\Omega_{\omega,c}^{+}.$ Next, by the mass, energy, momentum
conservation laws and the continuity of $K$ in $H^1(\R)$, we know that $
\Omega_{\omega,c}^{+}(\ni 0)$ is a closed subset of $I$. Last by the
uniform boundedness of $K_{\omega,c}$ in Lemma \ref{lem:Kunibound},
we know that $\Omega_{\omega,c}^{+}$ is open in $I$. Thus we have
$\Omega_{\omega,c}^{+} = I$.

Now for any $t\in I,$ we have
\begin{align*}
   \int\sts{\abs{\partial_x u\sts{t}}^2-\frac{3}{16}\abs{u\sts{t}}^6+\omega\abs{u\sts{t}}^2  - c\Im\sts{\overline{u\sts{t}}\partial_x u\sts{t}}+\frac{c}{2}\abs{u\sts{t}}^4
  }dx \geqslant 0,
\end{align*}
which implies that
\begin{align}
  &\int\sts{\abs{\partial_x u\sts{t}}^2 +\omega\abs{u\sts{t}}^2 - c\Im\sts{\overline{u\sts{t}}\partial_x u\sts{t}}}dx \notag\\
   \geqslant&  \int\sts{ \frac{3}{16}\abs{u\sts{t}}^6
   -\frac{c}{2}\abs{u\sts{t}}^4}dx \notag
   \\
=& 4\int\sts{ \frac{1}{32}\abs{u\sts{t}}^6
-\frac{c}{8}\abs{u\sts{t}}^4} + \frac{1}{16}\int\abs{u\sts{t}}^6dx
\notag
   \\
\geqslant& 4\int\sts{ \frac{1}{32}\abs{u\sts{t}}^6
-\frac{c}{8}\abs{u\sts{t}}^4}dx. \label{kbdd}
\end{align}
Note that for $t\in I$, we have
\begin{align*}
   J_{\omega,c}^0
   &>  J_{\omega,c}\sts{u\sts{t}}
   \\
   &=\int \left(\frac{1}{2}\abs{\partial_x u\sts{t}}^2  -\frac{1}{32}\abs{u\sts{t}}^6 +
  \frac{\omega}{2}\abs{u\sts{t}}^2 -
  \frac{c}{2}\Im\sts{\overline{u\sts{t}}\partial_x u\sts{t}}
  +\frac{c}{8}\abs{u\sts{t}}^4\right)dx \\
  & =  \frac{1}{2}\int\sts{\abs{\partial_x u\sts{t}}^2 +\omega\abs{u\sts{t}}^2 - c\Im\sts{\overline{u\sts{t}}\partial_x u\sts{t}}}  - \int \sts{ \frac{1}{32}\abs{u\sts{t}}^6  -\frac{c}{8}\abs{u\sts{t}}^4
  }dx.
  \end{align*}
  By \eqref{kbdd} and the Cauchy inequality, we obtain
  \begin{align*}
 J_{\omega,c}^0
  &\geqslant \frac{1}{4}\int\sts{\abs{\partial_x u\sts{t}}^2 +\omega\abs{u\sts{t}}^2 - c\Im\sts{\overline{u\sts{t}}\partial_x
  u\sts{t}}}dx
  \\
  &\geqslant \frac{1}{4}\sts{ \frac{1}{2}\int\abs{\partial_x u\sts{t}}^2dx +\sts{\omega-2c^2}\int\abs{u\sts{t}}^2dx   }
  \\
  &=\frac{1}{4}\sts{ \frac{1}{2}\int\abs{\partial_x u\sts{t}}^2dx +\sts{\omega-2c^2}\int\abs{u_0}^2dx
  },
\end{align*}
which implies $ \big\| u(t)\big\|_{\dot H^1} $ is uniformly bounded
on $I,$ thus $I=\R$ by the local wellposedness theory (Theorem
\ref{Thm:LWP}), which completes the proof.
\end{proof}

\begin{proof}[Proof of Corollary \ref{cor:gwp}]
By the assumptions, there exists some $c\gg 1$ such that for $M(u_0)\not = 0$,
 \begin{align*}
J_{c^2/4,c}(u_0) \triangleq &  E(u_0) +\frac{c^2}{4}  M(u_0) + c P(u_0)\\
   <  &  \frac{c^2}{4} 2\pi =J^0_{c^2/4,c}(\varphi_{c^2/4,c}),
\\
K_{c^2/4, c}(u_0) \triangleq &
   \int |\partial_xu_0|^2-\frac{3}{16}|u_0|^6+\frac{c^2}{4}|u_0|^2  - c\Im(\overline{u_0}\partial_x\varphi)+
   \frac{c}{2}|u_0|^4
   dx\\
   >& 0,
   \end{align*}
it implies that $u(0) \in \KKK^{+}_{c^2/4,c}$ for some $c\gg1$. Therefore we obtain the result by Theorem \ref{Thm:GWP}.
\end{proof}

\subsection*{Acknowledgements.} The authors would like to thank Takahisa Inui for introducing their work. The authors have been partially supported by the NSF grant of China (No. 11231006, No. 11671046, No. 11671047) and also partially
supported by Beijing Center of Mathematics and Information
Interdisciplinary Science.

%
%
%
%

\end{document}